\documentclass[a4paper,11pt,reqno]{amsart}

\usepackage[T1]{fontenc}
\usepackage[utf8]{inputenc}
\usepackage{lmodern}
\usepackage[english]{babel}

\usepackage[%
    a4paper,
	left=2.5cm,       
	right=2.5cm,      
	top=3.5cm,        
	bottom=3.5cm,     
	heightrounded,    
	bindingoffset=0mm 
]{geometry}

\usepackage{amsmath}
\usepackage{amssymb}
\usepackage{mathtools}
\usepackage{mathrsfs}
\usepackage[dvipsnames]{xcolor}
\usepackage{graphicx}
\usepackage{enumitem}

\usepackage{hyperref} 
\usepackage{esint}
\usepackage{verbatim}

\usepackage[noabbrev,capitalize]{cleveref}
\usepackage{bookmark}
\usepackage{autonum}

\usepackage[initials]{amsrefs}

\usepackage{bbm}

\usepackage[]{hyperref}
\hypersetup{
    colorlinks=true,       % false: boxed links; true: colored links
    linkcolor=red,          % color of internal links
    citecolor=blue,        % color of links to bibliography
    filecolor=red,      % color of file links
    urlcolor=cyan           % color of external links
}

\numberwithin{equation}{section}

\renewcommand{\phi}{\varphi}

\theoremstyle{definition}
\newtheorem{Def}{Definition}[section]

\theoremstyle{plain}
\newtheorem{prop}[Def]{Proposition}
\newtheorem{cor}[Def]{Corollary}
\newtheorem{thm}[Def]{Theorem}
\newtheorem{lem}[Def]{Lemma}

\theoremstyle{remark}
\newtheorem{rem}[Def]{Remark}

\DeclarePairedDelimiter{\abs}{|}{|}
\DeclarePairedDelimiter{\norm}{\|}{\|}

\DeclareMathOperator{\divergence}{div}

\newcommand{\e}{\varepsilon}
\newcommand{\R}{\mathbb{R}}
\newcommand{\N}{\mathbb{N}}
\newcommand{\W}{\mathcal{W}}
\newcommand{\Lip}{\text{Lip}}
\newcommand{\loc}{\text{loc}}
\newcommand{\meas}{\mathcal{L}}

\newcommand{\weakstar}{\overset{*}{\rightharpoonup}}

\newcommand{\Z}{\mathbb{Z}}
\newcommand{\C}{\mathcal{C}}
\newcommand{\Span}{\text{Span}}

\DeclareMathOperator{\di}{d\!}

\renewcommand{\MR}[1]{} % remove MR numbers

\newcommand{\marco}[1]{{\color{red} \textbf{M:} #1}}

\allowdisplaybreaks
\begin{document}

\title[Lagrangian solutions to the transport--Stokes system] {Lagrangian solutions to the transport--Stokes system}

\author[M.~Inversi]{Marco Inversi}
\address[M.~Inversi]{Department Mathematik und Informatik, Universität Basel, Spiegelgasse~1, 4051 Basel, Switzerland}
\email{marco.inversi@unibas.ch}



%\thanks{$^\ast$Corresponding author.}

\thanks{
\textit{Acknowledgements}.
MI is partially funded by the SNF grant FLUTURA: Fluids, Turbulence, Advection No. 212573. The author thanks Amina Mecherbet, Franck Sueur and Gianluca Crippa for several fruitful discussions on the topic. 
}

\begin{abstract}
In this paper we consider the transport--Stokes system, which describes the sedimentation of particles in a viscous fluid in inertialess regime. We show existence of Lagrangian solutions to the Cauchy problem with $L^1$ initial data. We prove uniqueness of solutions as a corollary of a stability estimate with respect to the $1$-Wasserstein distance for solutions with initial data in a Yudovich-type refinement of $L^3$, with finite first moment. Moreover, we describe the evolution starting from axisymmetric initial data. Our approach is purely Lagrangian. 
\end{abstract}

\date{\today}

\maketitle

\section{Introduction}

For a fixed $T>0$, we consider the so-called transport--Stokes system, which describes the sedimentation of inertialess suspension in a viscous flow
\begin{equation} \label{transport-stokes}
\tag{$\mathbf{TS}$}
\begin{cases}
\partial_t \rho + u \cdot \nabla \rho = 0 & (t,x) \in (0, T) \times \R^3, 
\\ -\Delta u + \nabla \pi = -  e_3 \rho & (t,x) \in (0, T) \times \R^3, 
\\ \divergence{u} = 0 & (t,x) \in (0, T) \times \R^3,  
\\ \displaystyle \lim_{\abs{x}\to +\infty} \abs{u(t,x)} = 0 & t \in (0,T),
\\ \rho(0,x)= \rho_0(x) & x \in \R^3.
\end{cases}
\end{equation} 
We refer to \cites{H18, M19} for a derivation of \eqref{transport-stokes} as a model for the sedimentation of a system of rigid particles in a viscous fluid in the regime where inertia of both fluid and particles is neglected. We also mention the recent paper\cites{HS23} where the authors show that a system of solid particles very small inertia converges in the mean field limit to the transport--Stokes system. The system \eqref{transport-stokes} couples a transport equation and the steady Stokes equations in the full three-dimensional space. Indeed, the density of particles $\rho$ is advected by an incompressible velocity field that solves the steady Stokes equation. 

Existence and uniqueness of classical solutions with a regular initial datum has been established in \cite{H18} and recently extended in \cite{M21} to the case of initial data in $L^1 \cap L^\infty(\R^3)$ with finite first moment. A similar result without the moment assumption has been proved in a parallel contribution by \cite{HS21}. We also mention the paper \cite{L22}, where global existence and uniqueness for bounded initial density is established in the case of bounded domain in $\R^3$ and $\R^2$, as well as in the infinite strip $(0,1)\times \R$ with an additional flux condition. We refer to \cite{G22} for a result on global existence and uniqueness for density with compact support in $L^1 \cap L^\infty(\R^2)$. In this paper, the authors study also the propagation of the H\"older regularity of the boundary patch. In the recent paper \cite{MS22}, it is considered the transport--Stokes system in $L^1\cap L^p(\R^3)$ for $p \geq 3$. More precisely, the authors prove existence and uniqueness of distributional solutions (see \cite{MS22}*{Theorem 2.1}), stability of solutions with respect to $1$-Wasserstein distance (see \cite{MS22}*{Theorem 2.2}) and analiticity of the trajectories (see \cite{MS22}*{Theorem 2.3}). Finally, we mention the paper \cite{C23} where the fractional transport--Stokes system is considered. 

In this paper, exploiting a purely Lagrangian approach, we study well-posedness of the transport-Stokes system \eqref{transport-stokes} beyond the results recently proved by Mecherbet and Sueur in \cite{MS22}. In particular, we adapt the technique of \cites{BBC16, BBC16bis} (where the authors deal with the 2D Euler equation and the Vlasov--Poisson system, respectively) to show the existence of Lagrangian solutions with initial density in $L^1(\R^3)$. Moreover, following the approach in \cite{CS21} for the 2D Euler equation, we establish a stability estimate with respect to the 1-Wasserstein distance for Lagrangian solutions with initial density in suitable refinements of $L^3(\R^3)$. Among other things, we point out the connection between several notions of solutions to \eqref{transport-stokes}. Finally, we exploit the invariance of the system under rotations along the vertical axis to describe the evolution corresponding to an initial density with singularities localized on the vertical axis, in the spirit of the recent paper \cite{DEL22}. 

Throughout this note, we work in the space-time domain $[0,T]\times \R^d$, mainly in the three dimensional case. Unless otherwise specified, we denote by $\norm{\cdot}_{L^p_t L^q_x}$ the norm in $L^p([0,T]; L^q(\R^d))$. 

\section{Main results}

The transport--Stokes system \eqref{transport-stokes} couples a transport equation with the steady Stokes system. Recall that the velocity field $u$ can be expressed in terms of the density $\rho$ by 
\begin{equation}
u=E*(-e_3 \rho), \ \ \ E(x) = \frac{1}{8 \pi \abs{x}} \left( Id + \frac{x}{\abs{x}} \otimes \frac{x}{\abs{x}}\right), \label{kernel} 
\end{equation}
where $E$ is the so-called Oseen tensor, namely the Green function for the Stokes flow (see \cite{G11}*{Theorem IV.2.1}). Thus, the transport--Stokes system \eqref{transport-stokes} shares the nonlinear transport structure with several equations arising in fluid mechanics, such as 2D Euler in vorticity formulation and Vlasov--Poisson. Hence, it is natural to consider solutions that are advected along the trajectories of the flow generated by the velocity field given by the Stokes system. 

\begin{Def} [Lagrangian solution] \label{Lagrangian solution}
We say that $\rho \in L^\infty([0,T]; L^1(\R^3))$ is a Lagrangian solution to the transport--Stokes system \eqref{transport-stokes} with initial condition $\rho_0 \in L^1(\R^3)$ if the following property holds true for any $t \in [0,T]$: 
\begin{equation}
\rho(t, X(t,0,x)) = \rho_0(x) \ \ \ \text{for a.e.\ } x \in \R^3, \label{representation of solution}
\end{equation}
where $X\colon[0,T]^2 \times \R^d \to \R^d$ is a regular Lagrangian flow associated to the velocity field $u= E*(-e_3 \rho)$ according to \cref{regular lagrangian flow} and $E$ is the Oseen tensor \eqref{kernel}. 
\end{Def}

The notion of Lagrangian solution is classical in the theory of the linear transport equation \cites{AC08, AC14} and it allows to describe active scalar equations such as \eqref{transport-stokes}. We refer to \cref{s:regular lagrangian flow} for further discussions on the notion of regular Lagrangian flow and, in particular, to \cref{sparse comments} for some basic properties of Lagrangian solutions to \eqref{transport-stokes} according to \cref{Lagrangian solution}. Motivated by \cites{BBC16, BBC16bis}, where similar results are established for the 2D Euler equation and the Vlasov--Poisson system, respectively, we state our main result on existence of Lagrangian solution to \eqref{transport-stokes}. 

\begin{thm} [Existence of solutions] \label{existence of sols}
Given $\rho_0 \in L^1(\R^3)$, there exists a Lagrangian solution $\rho \in L^\infty([0,T]; L^1(\R^3))$ to \eqref{transport-stokes} according to \cref{Lagrangian solution} with initial condition $\rho_0$.
\end{thm}

Exploiting the classical theory for the linear transport equation \cites{AC08, AC14}, we obtain the existence of distributional/renormalized solutions to \eqref{transport-stokes} (see the discussion in \cref{s:distr vs lagr}). 

\begin{cor} \label{distributionality}
Let $\rho_0 \in L^1(\R^3)$ and let $\rho \in L^\infty([0,T];L^1(\R^3))$ be any Lagrangian solution to \eqref{transport-stokes} with initial datum $\rho_0$ according \cref{Lagrangian solution}. Then, $\rho$ is a renormalized solution according to \cref{renormalized solution}. If $\rho_0 \in L^1 \cap L^p(\R^3)$ for some $p \geq 6/5$ and $\rho \in L^\infty([0,T];L^1 \cap L^p(\R^3))$ is a renormalized solution to \eqref{transport-stokes} with initial datum $\rho_0$ according to \cref{renormalized solution}, then $\rho$ is a distributional solution according to \cref{weak solution}. If $\rho_0 \in L^1 \cap L^p(\R^3)$ for some $p \geq 3/2$ and $\rho \in L^\infty([0,T];L^1 \cap L^p(\R^3))$ is a distributional solution to \eqref{transport-stokes} with initial condition $\rho_0$ according to \cref{weak solution}, then $\rho$ is a Lagrangian and renormalized solution according to \cref{Lagrangian solution} and \cref{renormalized solution}. 
\end{cor}

Following \cite{CS21}, we state a stability (thus uniqueness) result for Lagrangian solutions to \eqref{transport-stokes} in a Yudovich-type refinement of $L^3$. We refer to \cref{s:nonlinear superposition} for some comments on the method and the assumptions, as well as the precise definitions for the spaces and the functions involved in the statement.  

\begin{thm} [Stability estimate] \label{estimate on wasserstein}
Let $\Theta\colon [1, +\infty) \to [1, +\infty)$ be a non decreasing function. Let $\rho^1, \rho^2$ be two Lagrangian solutions to \eqref{transport-stokes} according to \cref{Lagrangian solution} with nonnegative initial conditions $ \rho_0^1, \rho_0^2$, respectively. Assume that  
\begin{equation}
    \norm{\rho_0^1}_{L^1_x} = \norm{\rho_0^2}_{L^1_x}, \label{same norm}
\end{equation}
\begin{equation}
    \int_{\R^3} \abs{x} \rho_0^1(x)\di x  + \int_{\R^3} \abs{x} \rho_0^2(x) \di x  < +\infty, \label{finite first moment}
\end{equation}
\begin{equation}
   \rho_0^1, \rho_0^2 \in L^\Theta(\R^3), \label{a priori growth} 
\end{equation}
\begin{equation}
    \omega_\Theta \text{ is concave,} \label{convexity}
\end{equation}
where $L^\Theta$ is the space in \cref{space theta} and $\omega_\Theta$ is given by \eqref{modulus of continuity}. If $\omega_\Theta$ satisfies the Osgood condition \eqref{osgood}, then there exists a function $\Gamma\colon [0,T] \times [0, +\infty)^3 \to [0, +\infty)$ depending only on $T, \Theta$ with the following properties: 
\begin{itemize}
    \item $\Gamma$ is non decreasing with respect to any variable; 
    \item $\Gamma (\cdot, \sigma, \cdot, \cdot)\to 0$ uniformly on compact sets of $[0,T] \times [0, +\infty)^2$ as $\abs{\sigma} \to 0$;
    \item there holds that 
    \begin{equation}
        \W_1(\rho^1(t),\rho^2(t)) \leq \Gamma (t, \W_1(\rho_0^1, \rho_0^2), \norm{\rho_0^1}_{L^\Theta_x}, \norm{\rho_0^2}_{L^\Theta_x}) \ \ \ \forall t \in [0,T]. \label{dis bella}
    \end{equation}
    where $\W_1$ denotes the $1$-Wasserstein distance (see \cref{s:wasserstein}). 
    \end{itemize}   
In particular, if $\rho_0^1 = \rho_0^2$, then for any $t \in [0,T]$ we have that $\rho^1(t, \cdot)=\rho^2(t, \cdot)$ for a.e.\ $x$ in $\R^3$.
\end{thm}

\begin{rem}
The interested reader can reconstruct the explicit formula for the function $\Gamma$ in \eqref{dis bella} by the proof of \cref{RLF for osgood} (see \eqref{explicit gamma} and \eqref{explicit gamma appendix}).
\end{rem}

In a well posedness regime (such as under the assumptions of \cref{estimate on wasserstein}), we exploit the structure of the equation to show that the cylindrical symmetry is preserved along the time evolution. In this symmetry regime, the vertical axis is invariant under the flow, namely the velocity field points downward on the vertical axis. In particular, we prove that, if the singularity at the initial time is on the vertical axis, the same property holds true for any positive time and the corresponding solution is bounded away from the vertical axis. We refer to \cite{DEL22}*{Theorem 3} for a similar result for the 2D Euler equation. We adopt the notation of \cref{cylinders and rotation}. 

\begin{thm} [Rotational invariant solutions] \label{symmetry regime}
Let $\Theta\colon[1,3) \to [1, +\infty)$ be a growth function such that $\omega_\Theta$ defined by \eqref{modulus of continuity} satisfies \eqref{convexity} and \eqref{osgood}. Given $\rho_0 \in L^\Theta(\R^3)$ such that \eqref{finite first moment} is satisfied, let $\rho \in L^\infty([0,T];L^\Theta(\R^3))$ be the unique Lagrangian solution to \eqref{transport-stokes} with initial condition $\rho_0$. Given an angle $\theta \neq 2k\pi$ for any $k \in \Z$, assume that $\rho_0$ is invariant under the rotation $R_\theta$ introduced in \cref{cylinders and rotation}. Then, the same holds for $\rho(t)$ for any $t \in [0,T]$. In particular, if $\rho_0$ has cylindrical symmetry, then the same holds for $\rho(t, \cdot)$ for any $t \in [0,T]$. Moreover, letting $\C_\delta$ be the cylinder introduced in \cref{cylinders and rotation}, if $\rho_0 \in L^p(\C_\delta^c)$ for some $\delta>0$ and $p \in [1, +\infty]$, then there exists $\e>0$ such that $\rho \in L^\infty([0,T]; L^p(\C_\e^c))$. 
\end{thm}

Finally, we give examples of axisymmetric initial data in $L^1 \cap L^{3-}(\R^3)$ but not in $L^3(\R^3)$ such that any Lagrangian solution to \eqref{transport-stokes} produced by \cref{existence of sols} satisfies the assumptions of \cref{symmetry regime}, with $\Theta(s) = 1- \log(3-s)$.  

\begin{prop} \label{explicit example}
Fix $\Theta(s) = 1- \log(3-s)$. Then, the function $\omega_\Theta$ defined by \eqref{modulus of continuity} satisfies \eqref{convexity} and \eqref{osgood}. Let us set 
\begin{equation}
    \rho_0(x) = \frac{\mathbf{1}_{B_{1/e}}(x)}{\abs{x} \abs{\log(x)}^{\frac{1}{3}}} \ \ \ x \in \R^3,  \label{rho_0}
\end{equation}
Then, $\rho_0 \in L^\Theta$ and \eqref{finite first moment} is satisfied. Let us also set
\begin{equation}
   \rho_1(y) = \frac{\mathbf{1}_{B_{1/e}}(y)}{\abs{y}^{\frac{2}{3}} \abs{\log(\abs{y})}^{\frac{1}{3}}} \ \ \ y \in \R^2. \label{tilde rho_0}
\end{equation} 
Then, for any function $\rho_2 \in L^3(\R)$ with compact support, we have that $\Tilde{\rho}_0(x) = \rho_1(x_1, x_2) \rho_2(x_3) \in L^\Theta(\R^3)$ and \eqref{finite first moment} is satisfied.  
\end{prop}

\section{Background materials}

\subsection{The regular Lagrangian flow} \label{s:regular lagrangian flow}

The notion of regular Lagrangian flow extends that of classical flow generated by smooth vector fields. We refer to the seminal papers \cites{DPL89, A04}, where a suitable notion of flow is considered (see also \cites{A08, AC08, AC14}). Throughout this note, we define the regular Lagrangian flow for incompressible vector fields following mainly \cites{CDL08, BC13}. A similar definition is available for vector fields with bounded divergence. 

\begin{Def} [Regular Lagrangian flow] \label{regular lagrangian flow}
Let $u\colon [0,T] \times \R^d \to \R^d$ be a Borel vector field in $L^1_{loc}([0,T]\times \R^d)$ such that $\divergence{u} =0$. Let $X\colon[0,T]^2 \times \R^d \to \R^d$ be a Borel map. We say that $X$ is a (complete) regular Lagrangian flow associated to $u$ if the following properties hold true for any $s\in [0,T]$: 
\begin{enumerate}
\item for any $t \in [0,T]$ the map $x \mapsto X(t,s,x)$ is a measure preserving transformation of $\R^d$; 
\item for a.e.\ $x \in \R^d$ the map $t \mapsto u(t, X(t,s,x))$ is in $L^1((0,T))$; 
\item for a.e.\ $x \in \R^d$ the map $t \mapsto X(t,s,x)$ is an absolutely continuous curve in $\R^3$ such that 
\begin{equation}
X(t,s,x) = x + \int_s^t u(z, X(z,s,x))\di z \ \ \ \forall t \in [0,T]. \label{ODE}
\end{equation}
\end{enumerate}
Given $s \in [0,T]$, we say that the map $X(\cdot, s, \cdot)$ is a regular Lagrangian flow starting at time $s$. 
\end{Def}

\begin{rem}
The notion of regular Lagrangian flow in \cref{regular lagrangian flow} agrees with that of classical flow within the Cauchy--Lipschitz theory. Whenever the vector field $u$ is globally bounded and Lipschitz continuous, there exists a classical flow $X\colon[0,T]^2 \times \R^d \to \R^d$. We also recall that the semigroup property holds: 
\begin{equation}
    X(t,\tau,X(\tau,s,x)) = X(t,s,x) \ \ \ \forall x \in \R^d, \ t,\tau,s \in [0,T] \nonumber. 
\end{equation}
In particular, we have that $X(t,s, \cdot)^{-1} = X(s,t,\cdot)$ for any $s,t \in [0,T]$. Indeed, the semigroup property is not included in \cref{regular lagrangian flow}, but it can be restored if stability estimates are available (see the discussion in \cite{BC13}*{Section 7}). Moreover, if $u$ is divergence free, by the classical Liouville Theorem, $X(t,s, \cdot)$ is a measure preserving bilipschitz transformation of $\R^d$ for any $t,s \in [0,T]$. In \cref{s:classical flow} we discuss how to extend these classical results to the case of vector field with modulus of continuity that satisfies the Osgood condition \eqref{osgood}. 
\end{rem}

In \cites{CDL08} the authors show quantitative estimates for the regular Lagrangian flow generated by a $W^{1,p}$ vector field (at least for $p>1$), implying existence, uniqueness and stability of the flow for weakly differentiable vector fields. These results are classical in the case $p= \infty$, i.e.\ Lipschitz vector fields. In \cite{BC13} the authors extend these results to the case of a vector field whose derivative is given by the convolution with a singular integral operator, covering also the case of $W^{1,1}$ vector fields. In \cite{BBC16tris}, the authors provide a stability estimate for the regular Lagrangian flow associated to a velocity field with anisotropic regularity conditions. 

\begin{rem} \label{summability of velocity}
Throughout this paper, we split the Oseen tensor \eqref{kernel} by
\begin{equation}
    E_1 = E \, \eta \in L^p(\R^3) \ \ \ \forall p \in [1, 3), \ \ \ E_2 = E \, (1-\eta) \in L^\infty(\R^3), \nonumber
\end{equation}
where $\eta$ is a nonnegative smooth cut-off function supported in $B_2$ and such that $\eta(x) =1$ for $x \in B_1$. Since $E = E_1 + E_2$, for any $\Phi \in L^1(\R^3)$ the convolution $E * \Phi$ is well defined in $L^p + L^\infty(\R^3)$ for any $p \in [1,3)$. Since for any $i,j,k =1,2,3$ we have that 
$$ \partial_k E_{i,j}(x) = \frac{1}{8 \pi} \left( -\frac{\delta_{i,j} x_k}{\abs{x}^3} + \frac{(\delta_{i,k} x_j + \delta_{j,k} x_i) \abs{x}^3 - 3 \abs{x} x_i x_j x_k}{\abs{x}^6}\right),  $$
by the chain rule, it is immediate to see that $\nabla E_1 \in L^p(\R^3)$ for any $p \in [1, 3/2)$ and $\nabla E_1 \in L^\infty(\R^3)$. Therefore, if $\Phi \in L^1(\R^3)$, we obtain that $E_1* \Phi \in W^{1,p}(\R^3)$ for any $p \in [1,3/2)$ and $E_2* \Phi \in W^{1,\infty}(\R^3)$ and with the following bounds: 
\begin{equation} \label{equibound}
\begin{split}
    \norm{E_1 * \Phi}_{L^p_x} & \leq C_p \norm{\Phi}_{L^1_x} \ \ \ p \in \left[1, 3 \right), 
    \\ \norm{\nabla E_1 * \Phi}_{L^p_x} & \leq C_p \norm{\Phi}_{L^1_x} \ \ \ p \in \left[1, 3/2 \right),  
    \\ \norm{E_2* \Phi}_{W^{1, \infty}_x} & \leq C \norm{\Phi}_{L^1_x}. 
\end{split} 
\end{equation}
Similarly, if $\Phi \in L^q(\R^3)$, we have that 
\begin{equation} \label{equibound 2} 
\begin{split}
    \norm{E_1 * \Phi}_{L^p_x} & \leq C_{p,q,r} \norm{\Phi}_{L^q_x} \ \ \ r \in [1,3), \ \frac{1}{r}+ \frac{1}{q} = 1+\frac{1}{p}, 
    \\ \norm{\nabla E_1 * \Phi}_{L^p_x} & \leq C_{p,q,r} \norm{\Phi}_{L^q_x} \ \ \ r \in [1,3/2), \ \frac{1}{r}+ \frac{1}{q} = 1+\frac{1}{p}, 
\end{split}
\end{equation}
We also remark that the second derivative of $E$ is a Calderon--Zygmund operator, thus defining by convolution a continuous operator in $L^p(\R^3)$ for any $p \in (1, +\infty)$, that is 
\begin{equation}
\norm{\nabla^2 E*\Phi}_{L^p_x}\leq C_p \norm{\Phi}_{L^p_x} \ \ \ p \in (1, +\infty). \label{equibound 4}
\end{equation} 
\end{rem}

We recall the following definition. 

\begin{Def} [Convergence in $L^0_\loc$] \label{convergence in measure}
Let $\{\phi_n\}_n , \phi$ be Borel functions on $\R^d$. We say that $\phi_n \to \phi$ in $L^0_{\loc}(\R^d)$ (locally in measure) if  
\begin{equation}
    \lim_{n \to +\infty} \meas^d(\{x \in B_R \ \colon  \ \abs{\phi_n(x) - \phi(x)} \geq \e \}) = 0 \ \ \ \forall R>0 ,\e >0. \nonumber
\end{equation}
Let $\{\Phi_n\}_n, \Phi$ be a Borel functions in $C([0,T]; L^0_{\loc}(\R^d))$. We say that $\Phi_n \to \Phi$ in $C(L^0_\loc)$ (uniformly in time locally in measure) if 
\begin{equation}
    \lim_{n \to +\infty}\sup_{t \in [0,T]} \meas^d(\{x \in B_R \ \colon  \ \abs{\Phi_n(t,x) - \Phi(t,x)} \geq \e \}) = 0 \ \ \ \forall R>0 ,\e >0. \nonumber
\end{equation}
\end{Def}

\begin{rem}
We recall that Cauchy sequences in $L^0_\loc(\R^3)$ (in $C([0,T];L^0_\loc(\R^3))$) are sequentially compact in $L^0_\loc(\R^3)$ (in $C([0,T];L^0_\loc(\R^3))$, respectively). 
\end{rem}

Motivated by the \cref{summability of velocity}, we recall the following apriori estimate on the regular Lagrangian flow associated to a velocity field $u \in L^\infty([0,T];W^{1,\infty} + W^{1,p}(\R^3))$, for $p \in [1,3/2)$. The stability of the regular Lagrangian flow is pivotal to build Lagrangian solution to \eqref{transport-stokes} according to \cref{Lagrangian solution}. We postpone the proof of the following lemma to \cref{s:stability estimate}.

\begin{lem} \label{quantitative RLF bis}
Let $K>0$ be a given constant. Take $\rho^1, \rho^2 \in L^\infty([0,T]; L^1(\R^3))$ such that 
\begin{equation}
\norm{\rho^1}_{L^\infty_t L^1_x} + \norm{\rho^2}_{L^\infty_t L^1_x} \leq K. \label{a priori condition}
\end{equation}
For $i,j=1,2$, we define $u^i_j = E_j* (-e_3\rho^i)$, where $E$ is the Oseen tensor \eqref{kernel}, with the notation of \cref{summability of velocity}. Then, for $i=1,2$ there exists a unique regular Lagrangian flow $X^i \in C([0,T]^2; L^0_\loc(\R^3))$ generated by $u^i$ according to \cref{regular lagrangian flow}. The regular Lagrangian flow $X^i$ enjoys the semigroup property \eqref{semigroup}. Moreover, for any $R>0, \e>0, \eta>0$ there exist a constant $C_{K, R, \e, \eta,T } >0$ and a radius $\lambda_{K,R,\e,\eta, T}>0$ such that 
\begin{equation}  \label{stability of RLF} 
    \begin{split}
        \sup_{t,s \in [0,T]} & \meas^d(\{ x \in B_R \ \colon  \ \abs{ X^1(t,s,x) - X^2(t,s,x)}> \e \}) 
    \\ & \leq C \left[ \norm{u^1_1-u^2_1}_{L^1_t L^1_x} + \norm{u^1_2-u^2_2}_{L^1_t L^1( B_\lambda ) } \right] + \eta. 
    \end{split}
\end{equation}
\end{lem}

\subsection{The 1-Wasserstein distance} \label{s:wasserstein}

We recall some notions in the theory of Optimal Transport. 

\begin{Def} [Transport plan] \label{transport plan}
Given $\mu, \nu$ nonnegative Borel measures on $\R^d$, we set
$$\Gamma(\mu, \nu) = \{ \gamma \in \mathcal{M}_+(\R^d \times \R^d) \ \colon  \ (\pi_1)_\# \gamma = \mu, \ (\pi_2)_\# \gamma = \nu \}, $$
where $\mathcal{M}_+$ is the collection of nonnegative Borel measures and $\pi_1, \pi_2\colon \R^d\times \R^d \to \R^d$ are the projection maps onto the first and the last $d$ coordinates, respectively. We say that $\gamma \in \Gamma(\mu, \nu)$ is a transport plan between $\mu$ and $\nu$. 
\end{Def}

\begin{Def} [$1$-Wasserstein distance] \label{wasserstein 1}
Let $\mu, \nu$ be nonnegative Borel measures on $\R^d$ with the same mass and finite first moment, i.e.\ 
\begin{equation}
\int_{\R^d} \abs{x} \di  \mu(x) + \int_{\R^d} \abs{x} \di  \nu(x) < +\infty. \label{finite moment}
\end{equation}
We define the $1$-Wasserstein distance between $\mu$ and $\nu$ by 
\begin{equation}
    \W_1(\mu, \nu) = \inf_{\gamma \in \Gamma(\mu, \nu)} \int_{\R^d \times \R^d} \abs{x-y} \di  \gamma(x,y) <+\infty, \nonumber
\end{equation}
where $\Gamma(\mu, \nu)$ is the collection of transport plans between $\mu$ and $\nu$ as in \cref{transport plan}.
\end{Def}

\begin{rem}
It is a classical result in the theory of Optimal Transport that the infimum in \cref{wasserstein 1} is achieved by an optimal transport plan (see \cite{S15}). 
\end{rem}

\begin{Def} [Optimal map] \label{optimal map}
Let $\rho^1, \rho^2 \in L^1(\R^d)$ be nonnegative functions such that $\norm{\rho^1}_{L^1_x} = \norm{\rho^2}_{L^1_x}$ with finite first moment, namely \eqref{finite moment} holds true. We say that $T\colon \R^d \to \R^d$ is an optimal map if the following conditions are satisfied: 
\begin{equation}
T_\# \rho^1 = \rho^2, \label{push forward}
\end{equation}
\begin{equation}
\W_1(\rho^1, \rho^2) = \int_{\R^d} \abs{x-T(x)} \rho^1(x)\di x , \label{minimality}
\end{equation}
where $\W_1$ is the Wasserstein distance as in \cref{wasserstein 1}. 
\end{Def}

\begin{rem} \label{maps vs plans}
Given $\rho^1, \rho^2$ as in \cref{optimal map} and a Borel map $T\colon \R^d \to \R^d$ satisfying \eqref{push forward}, then $T$ defines a transport plan between $\rho^1$ and $\rho^2$ by taking $\gamma = (Id \times T)_\# (\rho^1 \meas^d)$. The existence of optimal maps as in \cref{optimal map} is a classical result in the theory of Optimal Transport (see the discussion in \cite{S15}*{Section 3.1} and the references therein).  
\end{rem}

\subsection{On the notion of Lagrangian solution} \label{sparse comments}

Let $\rho \in L^\infty([0,T];L^1(\R^3))$ be a Lagrangian solution to \eqref{transport-stokes} according to \cref{Lagrangian solution} with initial condition $\rho_0 \in L^1(\R^3)$. \cref{Lagrangian solution} has to be interpreted taking into account the following remarks. 

\begin{rem} \label{norms are preserved}
Since the flow map $X(t,0, \cdot)$ is measure preserving for any $t$, by the push-forward formula and \eqref{representation of solution}, for any Borel function $\phi\colon \R\to [0, +\infty)$ we have that
\begin{align}
\int_{\R^3} \phi(\abs{\rho(t,x)}) \di x  & = \int_{\R^3} \phi(\abs{\rho(t, X(t,0,x))} ) \di x  = \int_{\R^3 } \phi(\abs{\rho_0(x)}) \di x  \ \ \ \forall t \in [0,T]. \label{L^1 norm of f}
\end{align}
In other words, any rearrangement invariant norm is preserved along the time evolution. Indeed, this properties holds true for Lagrangian solution to the linear transport equation. Moreover, since \eqref{transport-stokes} is a model for sedimentation of particles and the unknown $\rho$ is a density of mass, it is physically relevant to consider nonnegative solutions. Indeed, taking $\phi(s) = \mathbf{1}_{(-\infty, 0)}(s)$ in \eqref{L^1 norm of f}, one can check that if $\rho_0$ is a.e.\ nonnegative, then for any $t \in [0,T]$ we have that $\rho(t,\cdot) \geq 0 $ for a.e.\ $x \in \R^3$.
\end{rem}

\begin{rem}
In \cref{Lagrangian solution} we have fixed Borel representatives of $u, X(\cdot, 0, \cdot)$. Hence, these maps are pointwise defined. Moreover, since $X(t, 0,\cdot)$ is measure preserving for any $t \in [0,T]$, it is easy to check that if we modify $\rho, X(\cdot, 0, \cdot)$ on a negligible set in $[0,T]\times \R^3$, then the composition in \eqref{representation of solution} is affected only on a negligible set. Moreover, we notice that for any $s \in [0,T]$, the map $t \mapsto u(t, X(t,s,x))$ is in $L^1((0,T))$ for a.e.\ $x \in \R^d$. Indeed, writing $u=u_1+u_2$, with $u_1 \in L^\infty([0,T];L^1(\R^3))$ and $u_2 \in L^\infty([0,T]; L^\infty(\R^3))$ as in \cref{summability of velocity} and since $X(t, s, \cdot)$ is measure preserving for any $t$, we have that
\begin{align}
\int_{\R^3}  \int_0^T \abs{ u_1(z, X(z,s,x))}\di z \di x    =  \int_{0}^T \int_{\R^3} \abs{u_1(z,x)}\di x  \di z = \norm{u_1}_{L^\infty_t L^1_x},  
\end{align} 
yielding that $u_1(\cdot, X(\cdot,s,x)) \in L^1((0,T))$ for a.e.\ $x \in \R^3$. Similarly, we have that
$$\norm{u_2(\cdot,X(\cdot, s, \cdot))}_{L^\infty_t L^\infty_x} \leq \norm{u_2}_{L^\infty_t L^\infty_x }.$$
Thus, for a.e.\ $x \in \R^3$, we have that $u_2(\cdot, X(\cdot,s, x))\in L^\infty((0,T))$. 
\end{rem}

\begin{rem} 
By \cref{quantitative RLF bis} the regular Lagrangian flow associated to the velocity field $u = E *(-e_3 \rho)$ satisfies the semigroup property \eqref{semigroup}. Therefore, \eqref{representation of solution} is equivalent to require that the following property holds for any $t \in [0,T]$:  
\begin{equation}
    \rho(t,x) = \rho_0(X(0,t,x)) \ \ \ \text{for a.e.\ } x \in \R^3. \label{representation of solution bis}
\end{equation}
\end{rem}

\section{Existence of Lagrangian solutions}

Our proof of \cref{existence of sols} is based on the same technique as \cites{BBC16, BBC16bis} for the 2D Euler equation and Vlasov--Poisson system, respectively. The strategy can be summarized as follows. We find a family of regular initial data that approaches the initial density. Then, exploiting the well-posedness of \eqref{transport-stokes} for regular initial data (see \cite{MS22}*{Theorem 2.1}), we may consider the family of solutions corresponding to the regularized initial data. Hence, using the stability property of the regular Lagrangian flow (see \cref{s:regular lagrangian flow}), we prove compactness at the level of the flows of the regular solutions and we produce a limiting flow advecting the limiting initial datum. 

\subsection{Building a Lagrangian solution}

The following result is a general statement within the Cauchy--Lipschitz theory, whose simple proof can be found in \cite{BBC16}*{Lemma 4.3}, for instance. We claim that weak convergence of vector fields gives strong convergence of the associated flows. 

\begin{lem} \label{weakly convergent vector fields}
Let $b^n, b$ be vector fields uniformly bounded in $L^\infty((0,T)\times \R^d)$ with $\nabla_x b^n, \nabla_x b$ uniformly bounded in $L^\infty((0,T)\times \R^d)$, such that $b^n \weakstar b $ in $L^\infty((0,T)\times \R)$. Letting $X^n,X\colon[0,T]^2 \times \R^d\to \R^d$ be the flows of $b^n, b$, respectively, then $X^n \to X$ in $C_\loc ([0,T]^2 \times \R^d)$. 
\end{lem}

With the help of the stability estimates of \cref{quantitative RLF bis}, \cref{weakly convergent vector fields} extends to vector fields with Sobolev regularity. The following result is similar to \cite{BBC16}*{Proposition 4.4}. 

\begin{prop} \label{weakly convergent vector fields bis}
Let $\{\rho^n\}_n$ be a sequence in $L^\infty([0,T]; L^1(\R^3))$ such that 
$$\sup_{n \in \N} \norm{\rho^n}_{L^\infty_t L^1_x} \leq K < +\infty.$$
Let $X^n$ be the regular Lagrangian flow associated to $u^n = E* (- e_3 \rho^n)$, where $E$ is the Oseen tensor \eqref{kernel}. Then, there exist a subsequence (not relabelled) and an incompressible vector field $u= u_1+u_2$ with the following properties: 
\begin{itemize}
\item $u_1 \in L^\infty([0,T]; W^{1,p}(\R^3))$ for $p \in [1, 3/2)$ and $u_2 \in L^\infty([0,T]; W^{1,\infty}(\R^3))$; 
\item $u^n_1 \weakstar u_1 $ in $L^\infty([0,T]; W^{1,p}(\R^3))$ for any $p \in [1, 3/2)$; 
\item $u^n_2 \weakstar u_2$ in $L^\infty([0,T]; W^{1,\infty}(\R^3))$.
\end{itemize}
Moreover, letting $X^n,X$ be the regular Lagrangian flows associated to $u^n,u$ respectively, then $X_n \to X$ in $C([0,T]^2; L^0_\loc(\R^3))$.  
\end{prop}

\begin{proof}
Since $\{\rho^n\}_n$ is bounded in $L^\infty([0,T];L^1(\R^3))$, by \cref{summability of velocity} we get that $\{u^n_1\}$ is bounded in $L^\infty([0,T];W^{1,p}(\R^3))$ for $p \in [1, 3/2)$ and $\{u^n_2\}_n$ is bounded in $L^\infty([0,T];W^{1,\infty}(\R^3))$. Therefore, we find a limit vector field $u$ and a subsequence (not relabelled) such that the properties above hold true. It remains to prove the convergence of the flows. Fix a Friedrichs' mollifier $j_\delta$ and regularize the velocity fields with respect to the spatial variable. Thus, we set
$$u^n_{i,\delta} = u^n_i * j_\delta, \ \ \ u_{i,\delta} = u_i*j_{\delta} \ \ \ i = 1,2,\ n\in \N, \ \delta>0. $$
Then, fix $\delta>0$ and notice that the sequence $\{u^n_\delta\}_n$ satisfies the assumptions of \cref{weakly convergent vector fields}. Hence, letting $X^n_\delta, X_\delta$ be the classical flows associated to $u^n_\delta, u_\delta$ respectively, we infer that $X^n_\delta \to X_\delta$ in $C_\loc([0,T]^2 \times \R^3)$. 
In particular, if we write
$$X^n - X = (X^n - X^n_\delta) + (X^n_\delta - X_\delta) + (X_\delta - X) = I + II + III, $$
then for any $\delta>0$ we have that $X^n_\delta -X^n \to 0$ in $C([0,T]^2; L^0_\loc(\R^3))$. To estimate the first term, by the properties of convolution, we infer that 
\begin{equation}
u^n_{i,\delta} = u^n_i * j_\delta = (E_i * (-e_3 \rho^n ) )* j_\delta = E_i * (- e_3 \rho^n * j_\delta ) \ \ \ i=1,2. \nonumber
\end{equation}
Noticing that 
$$\sup_{n \in \N, \delta >0} \norm{\rho^n * j_\delta}_{L^\infty_t L^1_x} \leq K <+\infty$$
by \cref{quantitative RLF bis}, for any $R>0, \e>0, \eta>0$ we find a constant $C_{R, \e, \eta,T, K}>0$ and a radius $\lambda_{R, \e, \eta,T, K}>0$ (both independent of $n$ and $\delta$) such that \eqref{stability of RLF} holds true with $X^1 = X^n$ and $X^2 = X^n_\delta$. Then, we estimate the right hand side of \eqref{stability of RLF}. Writing explicitly the convolutions we estimate $\norm{u^n_{1,\delta}-u^n_1}_{L^1_t L^1_x}$: 
\begin{align}
    & \norm{u^n_{1,\delta} - u^n_1}_{L^1_t L^1_x}  \leq T \sup_{t \in [0,T]} \int_{\R^3} \bigg| \int_{B_\delta} (u^n_{1}(t,x-y) - u^n_1(t,x)) j_\delta(y) \di y \bigg| \di x  
    \\ & \quad \leq T \sup_{t \in [0,T]} \int_{B_\delta} j_\delta(y) \int_{\R^3 } \abs{u^n_1(t,x-y) - u^n_1(t,x)}\di x  \di y
    \\ & \quad = T \sup_{t \in [0,T]} \int_{B_\delta} j_\delta(y) \int_{\R^3} \bigg| \int_{\R^3}(E_1(x-y-z) - E_1(x-z))(-e_3 \rho^n(t,z))\di z\bigg| \di x  \di y 
    \\ & \quad \leq T \sup_{t \in [0,T]} \int_{B_\delta} j_\delta(y) \int_{\R^3}\abs{\rho^n(t,z)} \int_{\R^3}\abs{E_1(x-y-z) - E_1(x-z)} \di x  \di z \di y 
    \\ & \quad = T \sup_{t \in [0,T]} \int_{B_\delta} j_\delta(y) \int_{\R^3}\abs{\rho^n(t,z)} \int_{\R^3}\abs{E_1(x-y) - E_1(x)} \di x  \di z \di y  
    \\ & \quad = T \norm{\rho^n}_{L^\infty_t L^1_x} \int_{B_\delta}   j_\delta(y) \alpha(y) \di y 
\end{align}
where we set 
$$\alpha(y) = \int_{\R^3} \abs{E_1(x-y)- E_1(x)}\di x . $$
Since $E_1 \in L^1(\R^3)$, we infer that $\alpha(y)\to 0$ as $y \to 0$, yielding 
\begin{equation}
    \lim_{\delta \to 0} \sup_{n \in \N} \norm{u^n_{1,\delta} - u^n_1}_{L^1_t L^1_x} \leq T K \lim_{\delta \to 0} \int_{B_\delta} j_\delta(y) \alpha_1(y)\di y = 0. \label{weakly conv 1} 
\end{equation}
With the same technique, we estimate $\norm{u^n_{2,\delta}-u^n_2}_{L^\infty_t L^1(B_\lambda)}$: 
\begin{align}
    \norm{u^n_{2,\delta} - u^n_2}_{L^1_t L^1(B_\lambda)} & \leq T \sup_{t \in [0,T]} \int_{B_\delta} j_\delta(y) \int_{\R^3}\abs{\rho^n(t,z)} \int_{B_\lambda} \abs{E_2(x-y-z) - E_2(x-z)} \di x  \di z \di y 
    \\ & \leq T \Lip(E_2) \lambda^3 \norm{\rho^n}_{L^\infty_t L^1_x} \int_{B_\delta} j_\delta(y)\abs{y}\di y, 
\end{align}
since $E_2$ is Lipschitz continuous, yielding 
\begin{equation}
    \lim_{\delta \to 0} \sup_{n \in \N} \norm{u^n_{2,\delta} - u^n_2}_{L^1_t L^1(B_\lambda)} \leq T \Lip(E_2) \lambda^3 K \lim_{\delta \to 0} \int_{B_\delta} j_\delta(y) \abs{y}\di y = 0. \label{weakly conv 2} 
\end{equation}
Thus, we have shown that $X^n_\delta - X^n \to 0$ in $C([0,T]^2; L^0_\loc(\R^d))$ as $\delta \to 0$, uniformly with respect to $n \in \N$. Regarding the third term, since $u_{1,\delta}\to u_1$ in $L^1([0,T]; L^1(\R^3))$ and $u_{2,\delta} \to u_2$ in $L^1([0,T]; L^1_\loc(\R^3))$, by \cref{quantitative RLF} (recall that we still do not know that $u= E*(- e_3 \rho)$ for some $\rho \in L^\infty([0,T]; L^1(\R^3))$), we obtain that $X_\delta - X \to 0$ in $C([0,T]^2; L^0_\loc(\R^3))$ as $\delta \to 0$. Therefore, we conclude that $X^n - X \to 0$ in $C([0,T]^2; L^0_\loc(\R^3))$ as $n \to +\infty$. 
\end{proof}

Finally, we prove \cref{existence of sols}, following the lines of \cite{BBC16}*{Theorem 5.1}.   

\begin{proof}[Proof of \cref{existence of sols}]
Given $\rho_0 \in L^1(\R^3)$, we fix a Friedrichs' mollifier $j_\delta$ and consider the classical solution $\rho_\delta$ to \eqref{transport-stokes} with smooth initial datum $\rho_{0,\delta} = \rho_0 * j_\delta$ (see \cite{MS22}). Letting $u_\delta = E*(- e_3 \rho_\delta)$ and denoting by $X_\delta$ the classical flow generated by $u_\delta$, we have that $\rho_\delta$ is transported by $X_\delta$. In other words, as explained in \eqref{representation of solution bis}, for any $t,x,\delta$ we have that $\rho_\delta(t,x) = \rho_{0,\delta}(X_\delta(0,t,x))$. Since $\{\rho_\delta\}_{\delta}$ is a bounded sequence in $L^\infty([0,T]; L^1(\R^3))$, we find a subsequence (not relabelled) and an incompressible vector field $u$ with the properties of \cref{weakly convergent vector fields bis}. Moreover, denoting by $X$ the regular Lagrangian flow generated by $u$, we have that $X_\delta \to X$ in $C([0,T]^2; L^0_\loc(\R^3))$. Letting $\rho(t,x) = \rho_0(X(0,t,x))$, we claim that $\rho$ is a Lagrangian solution to \eqref{transport-stokes} with initial condition $\rho_0$. Indeed, \eqref{representation of solution} holds true by \eqref{representation of solution bis}. Moreover, we only have to check that $u= E*(-e_3 \rho)$. By the same argument of \cite{BC12}*{Proposition 7.7}, we obtain that $\rho_\delta \to \rho$ in $C([0,T]; L^1(\R^3))$, yielding $u_\delta = E*(- e_3 \rho_\delta) \to E*(-e_3 \rho)$ in $C([0,T]; L^1_\loc(\R^3))$. Since $u_\delta\weakstar u$ in $\mathcal{D}'((0,T) \times \R^3)$, we conclude that $u= E*(-e_3 \rho)$. 
\end{proof}

\subsection{Lagrangian vs weak vs renormalized solutions} \label{s:distr vs lagr}

In this section, we discuss the relation between Lagrangian, distributional and renormalized solutions to \eqref{transport-stokes}. This is a classical topic in the theory of the of linear transport equation \cites{AC08, AC14}. Indeed, in order to give \eqref{transport-stokes} a distributional meaning, we need that $\rho u \in L^1_\loc([0,T] \times \R^3)$. Therefore, if we assume $\rho \in L^\infty([0,T]; L^1 \cap L^p(\R^3) )$ for some $p \in (1, +\infty)$, recalling the decomposition $E= E_1+E_2$ by \cref{summability of velocity}, we have that $E_2*(-e_3 \rho) \in L^\infty([0,T]; L^\infty(\R^3))$ (see \eqref{equibound}). To estimate the convolution with $E_1$, recall that the second derivative of the Oseen tensor $E$ is a Calderon--Zygmund operator, yielding $E_1 *(-e_3 \rho) \in L^\infty([0,T]; W^{2,p}(\R^3))$. Hence, for $p \geq 3/2$, by Sobolev embedding we deduce that $\rho u \in L^1_\loc([0,T] \times \R^3)$. For $p \in (1, 3/2)$, by Sobolev embedding we have that $u \in L^\infty([0,T]; L^{q}_\loc(\R^3))$, where $q= 3p/(3-2p)$. Thus, $\rho u \in L^1_\loc([0,T] \times \R^3) $ provided that 
$$\frac{1}{p} + \frac{3-2p}{3p} \leq 1 \Longleftrightarrow p \geq \frac{6}{5}. $$
In other words, the distributional formulation of \eqref{transport-stokes} is available only in the range $p \in [6/5, +\infty]$. Thus, to study \eqref{transport-stokes} with $L^1$ initial data, the Lagrangian formulation is the only appropriate one. Motivated by the discussion above, we recall the following definitions. 

\begin{Def} [Distributional solution] \label{weak solution}
Fix $p \geq 6/5$. Given $\rho \in L^\infty([0,T]; L^1\cap L^p(\R^3))$, let us set $u= E*(-e_3 \rho)$, where $E$ is the Oseen tensor \eqref{kernel}. We say that $\rho$ is distributional solution to \eqref{transport-stokes} with initial datum $\rho_0 \in L^1 \cap L^p(\R^3)$ if
\begin{equation}
    \int_0^T \int_{\R^3} (\partial_t \phi + u \cdot \nabla \phi) \rho \di x  \di t = -\int_{\R^3} \rho_0(x) \phi(0,x)\di x \ \ \ \forall \phi \in C^\infty_c([0,T)\times \R^3). \label{weak formulation}
\end{equation}
\end{Def}

\begin{Def} [Renormalized solution] \label{renormalized solution}
Given $\rho \in L^\infty([0,T]; L^1(\R^3))$, let us set $u= E*(-e_3 \rho)$, where $E$ is the Oseen tensor \eqref{kernel}. We say that $\rho$ is a renormalized solution to \eqref{transport-stokes} with initial datum $\rho_0 \in L^1(\R^3)$ if for any bounded function $\beta \in C^1(\R, \R)$ with $\beta'$ bounded there holds that
\begin{equation}
    \int_0^T \int_{\R^3} (\partial_t \phi + u \cdot \nabla \phi) \beta(\rho) \di x  \di t = -\int_{\R^3} \beta(\rho_0(x)) \phi(0,x)\di x  \ \ \ \forall \phi \in C^\infty_c([0,T)\times \R^3). \label{renormalized formulation}
\end{equation} 
\end{Def}

If $p\geq 3/2$, by the DiPerna--Lions theory \cite{DPL89}, weak solutions to \eqref{transport-stokes} are renormalized according to \cref{renormalized solution}. Indeed,  given $\rho \in L^\infty([0,T]; L^p(\R^3))$ for some $p \in (1, +\infty)$, by Calderon--Zygmund estimates and Sobolev embedding, we have that $u= E*(-e_3 \rho) \in L^\infty([0,T]; W^{1,p^*}_\loc(\R^3))$, where $p^* = 3p/(3-p)$ if $p<3$, $p^*$ is any exponent in $(1, +\infty)$ for $p =3$ and $p^*= \infty$ for $p > 3$. Thus, $(\rho, u)$ satisfies the Di Perna--Lions condition if 
$$ \frac{1}{p} + \frac{3-p}{3p} \leq 1 \Longleftrightarrow p \geq \frac{3}{2}.  $$ 

We discuss the proof of \cref{distributionality}. 

\begin{proof} [Proof of \cref{distributionality}]
Fix a bounded scalar function $\beta \in C^1(\R)$ with $\beta'$ bounded and a test function $\phi \in C^\infty_c([0,T) \times \R^3)$. Since $\beta(\rho) u \in L^1_\loc([0,T]\times \R^3)$, by the push-forward formula and \eqref{representation of solution}, we have that 
\begin{equation} \label{distributional formulation 1}
    \begin{split}
         \int_0^T \int_{\R^3} & \beta(\rho(t,x)) ( \partial_t \phi(t,x) + u(t,x)\cdot \nabla \phi(t,x)) \di x  \di t 
    \\ & = \int_0^T \int_{\R^3} \beta(\rho_0(x)) ( \partial_t \phi(t,X(t,0,x)) + u(t,X(t,0,x))\cdot \nabla \phi(t,X(t,0,x))) \di x  \di t 
    \\ & = \int_{\R^3}  \beta(\rho_0(x)) \left[ \int_0^T ( \partial_t \phi(t,X(t,0,x)) + u(t,X(t,0,x))\cdot \nabla \phi(t,X(t,0,x))) \di t \right] \di x . 
    \end{split}
\end{equation}
Moreover, given $x \in \R^3$ such that the trajectory $t \mapsto X(t,0,x)$ is an absolutely continuous curve that satisfies \eqref{ODE}, then by the chain rule for Sobolev functions we get that $t \mapsto \phi(t, X(t,0,x))$ is in $W^{1,1}((0,T))$ and there holds that 
\begin{equation} \label{distributional formulation 2}
    \begin{split}
         - \phi(0, x) & = \phi(T, X(T,0,x)) - \phi(0, X(0,0,x)) 
    \\ & = \int_0^T \left[ \partial_t \phi(t, X(t,0,x)) + \dot{X}(t,0,x) \cdot \nabla \phi(t, X(t,0,x)) \right] \di t 
    \\ & = \int_0^T \left[ \partial_t \phi(t, X(t,0,x)) + u(t, X(t,0,x)) \cdot \nabla \phi(t, X(t,0,x)) \right] \di t. 
    \end{split}
\end{equation}
Integrating \eqref{distributional formulation 2} with respect to the measure $\beta(\rho_0)\mathcal{L}^3$, then \eqref{renormalized formulation} follows by \eqref{distributional formulation 1}. 

Given $p \geq 6/5$, let $\rho_0 \in L^1 \cap L^p(\R^3)$ and let $\rho \in L^\infty([0,T]; L^1 \cap L^p(\R^3))$ be a renormalized solution to \eqref{transport-stokes} according to \cref{renormalized solution} with initial condition $\rho_0$. Let $\{\beta_n\}_n$ a sequence of scalar functions with the following properties: 
\begin{itemize}
    \item $\beta_n \in C^1(\R)$ is bounded and $\beta'_n$ is bounded; 
    \item $\beta_n(s) = s$ for $s \in [-n,n]$; 
    \item $\abs{\beta_n(s)} \leq \abs{s}$ for any $n\in \N, s \in \R$. 
\end{itemize}
Fix a test function $\phi \in C^\infty_c([0,T)\times \R^3)$. Writing \eqref{renormalized formulation} for $\beta_n$ and letting $n \to \infty$, in the limit we obtain \eqref{weak formulation}. To be precise, we can pass to the limit in \eqref{renormalized formulation} by the Dominated Covergence Theorem, since $\rho u \in L^1_\loc([0,T]\times \R^3)$, as discussed at the beginning of this section. 

Lastly, fix $p \geq 3/2$ and let $\rho \in L^\infty([0,T]; L^p(\R^3))$ be any distributional solution to \eqref{transport-stokes} with initial datum $\rho_0$. As discussed above, we have that $u \in L^\infty([0,T]; W^{1,q}_\loc(\R^3))$ and $p,q$ satisfies $1/p + 1/q \leq 1 $. 
Hence, $\rho$ is the unique weak solution to the linear transport equation with velocity field $u$. Then, by uniqueness, we infer that $\rho$ is a Lagrangian solution to \eqref{transport-stokes} according to \cref{Lagrangian solution}. In particular, $\rho$ solves \eqref{transport-stokes} in the renormalized sense (see \cref{renormalized solution}). 
\end{proof}

\section{Stability of Lagrangian solutions}

This section is devoted to the proof of \cref{estimate on wasserstein}. 

\subsection{A nonlinear superposition-type principle} \label{s:nonlinear superposition}

The superposition principle is a classical tool in the theory of the linear transport equation, that can be summarized as follows. Uniqueness of the trajectories at the level of the ODE with a given vector field is essentially equivalent to uniqueness of nonnegative solutions at the level of the linear transport equation with the same vector field. However, exploiting the linear structure, it is possible to prove this substantial equivalence in a very low regularity setting. We refer to \cites{A04, AC08, AC14} for an extensive discussion in the linear case. 

As already mentioned, the transport--Stokes system share some structural properties with the 2D Euler equation and the Vlasov--Poisson system, the only difference being the convolution kernel that provides the velocity field from the advected solution. In the case of 2D Euler, the so-called Biot--Savart law maps a vorticity in $L^\infty([0,T]; L^1\cap L^p(\R^2))$ onto a velocity in $L^\infty([0,T]; W^{1,p}_\loc(\R^2))$, as soon as $p< +\infty$. Due to lack of Calderon--Zygmund estimates in $L^\infty$, the velocity $u$ associated to a vorticity $\omega \in L^\infty([0,T]; L^1 \cap L^\infty(\R^2))$ belongs to $L^\infty([0,T]; W^{1,p}_{\loc}(\R^2))$ for any $p< +\infty$, enjoying a log-Lipschitz modulus of continuity in space, uniformly in time. Therefore, $u$ has a classically defined flow (see the discussion in \cref{s:classical flow}) and the argument by Yudovich \cite{Y63} shows that uniqueness for the ODE associated to the linear transport equation is enough to have uniqueness at the level of the nonlinear PDE. In \cite{Y95}, Yudovich extends his previous result to the case of a vorticity in any $L^p$ for $p$ finite, with a logarithmic growth of the $L^p$ norms as $p \to +\infty$. Thus, the corresponding velocity field turns out to be Osgood continuous, allowing for a classically defined flow. As in \cite{Y63}, uniqueness of trajectories at the level of the ODE associated to the linear transport equation is enough to prove uniqueness of Lagrangian solutions to the nonlinear PDE. Following this principle, a further generalization of the uniqueness result in \cite{Y95}, has recently been established in \cite{CS21}. Among other things, the authors prove uniqueness for solutions in $L^\infty([0,T]; L^p(\R^2))$ for any $p$ finite, provided that the $L^p$ norms grow slow enough to give a velocity field with modulus of continuity that satisfies the Osgood condition. Broadly speaking, we denote these refinements of $L^\infty$ as Yudovich-type spaces. To summarize, the aforementioned uniqueness results on the 2D Euler equations \cites{Y63, Y95, CS21} is based on the following nonlinear superposition-type principle: quantitative stability (thus uniqueness) for the trajectories of the ODE associated to the linear transport equation is enough to prove stability (thus uniqueness) at the level of Lagrangian solution for the nonlinear transport equation. 

The method of \cite{CS21} relies on basic techniques and extends to other active scalar equations, such as the Vlasov--Poisson system. In this case, we refer to \cite{L06} for the case of bounded density and \cite{M14} for the case of density in any $L^p$ space for $p$ finite, with $p \log(p)$ growth of the $L^p$ norms as $p \to \infty$. We also mention the upcoming paper \cite{CISS23}, where the authors, following \cite{CS21}, exploit the Hamiltonian structure of the associated ODE system to find general conditions on the growth of $L^p$ norms of the density that allow for a classically defined flow. Lastly, in \cites{CL13,IS23} the authors consider continuity equations of non-local type under continuity assumptions on the convolution kernel. 

The uniqueness and stability for the transport--Stokes system can be studied within the same framework of \cite{CS21}. Indeed, since the convolution with the Oseen tensor \eqref{kernel} allows for a gain of two derivatives (see \cref{summability of velocity}), the space $L^3$ for transport--Stokes plays the role of the critical space $L^\infty$ for 2D Euler. Indeed, by Calderon--Zygmund inequality and Sobolev embedding, the velocity field associated to a density in $L^{3+}$ is Lipschitz continuous, allowing for a classically defined flow. In the case of $L^3$ density, the corresponding velocity field is log-Lipschitz continuous, thus possessing a classical flow. Therefore, it is relevant to look for Yudovich-type refinements of $L^3$. Motivated by the discussion above, we introduce some notation. 

\begin{Def} \label{space theta}
Let $\Theta\colon [1,3) \to [1, +\infty)$ be a non decreasing function. We define   
$$L^\Theta(\R^3) = \left\{ u \in \bigcap_{p \in [1,3)} L^p(\R^3) \ : \ \norm{u}_{L^\Theta_x} = \sup_{p \in [ 1, 3)}  \frac{ \norm{u}_{L^p_x}}{\Theta(p)} <+\infty \right\}. $$
We say that $L^\Theta$ is a Yudovich-type refinement of $L^3$ with growth function $\Theta$. 
\end{Def}

\begin{Def} 
Let $\Theta\colon [1,3) \to [1, +\infty)$ be a non decreasing function. We define 
\begin{equation}
\omega_\Theta(s) = 
\begin{cases}
\displaystyle s (1-\log(s)) \Theta\left( \frac{2-3\log(s)}{1- \log(s)}\right) & s \in (0, 1),
\\ \Theta(2) & s \in [1, +\infty). 
\end{cases} \label{modulus of continuity}
\end{equation}
\end{Def}

\begin{Def} [Osgood condition]
Let $\omega\colon [0, +\infty) \to [0, +\infty)$ be a modulus of continuity, i.e.\ $\omega$ is a non decreasing function such that $\omega(s)=0$ if and only if $s=0$ and $\omega(s)\to 0$ as $s \to 0$. We say that $\omega$ satisfies the Osgood condition if 
\begin{equation}
    \int_0^1 \frac{1}{\omega(s)} \di s = +\infty. \label{osgood}  
\end{equation}
\end{Def}

\begin{Def} \label{C^Theta}
Let $\Theta\colon [1, +\infty) \to  [1, +\infty)$ be a non decreasing function and let $\omega_\Theta$ be defined by \eqref{modulus of continuity}. We denote by
\begin{equation}
C^{\omega_\Theta} (\R^d) = \left\{ g\colon \R^d\rightarrow \R \ : \ \sup_{x\neq y} \frac{\abs{g(x)- g(y)}}{\omega_\Theta (\abs{x-y})} < +\infty \right\}. 
\end{equation}
\end{Def}

\begin{rem}
If $\Theta\colon [1,3) \to [1, +\infty)$ is a bounded function, then we have that $L^\Theta(\R^d) = L^1 \cap L^3(\R^d)$, by \cref{space theta}. Hence, throughout this note, we focus on the case in which $\Theta$ is unbounded. The fact that $\omega_\Theta$ in \eqref{modulus of continuity} is a modulus of continuity and the validity of \eqref{osgood} depend on the rate of blow-up of $\Theta$ at $3^-$.  
\end{rem}

\begin{rem}
Under the assumptions of \cref{estimate on wasserstein}, $\rho^1(t), \rho^2(t)$ are positive densities, with the same mass and finite first moment, for any $t \in [0,T]$. Hence, we can compare $\rho^1(t), \rho^2(t)$ with respect to the $1$-Wasserstein distance (see \cref{wasserstein 1}). Indeed, as explained in \cref{norms are preserved}, \eqref{same norm} and the sign are propagated along the time evolution by Lagrangian solutions to \eqref{transport-stokes}. Moreover, by \eqref{equibound}, \eqref{equibound 2} we infer that $u^i \in L^\infty([0,T];L^\infty(\R^3))$ for $i=1,2$. Thus, by the push-forward formula and \eqref{representation of solution}, for $i=1,2$ and for any $t \in [0,T]$ we have that 
\begin{align}
    \int_{\R^3} \abs{x} \rho^i(t,x)\di x  & = \int_{\R^3} \abs{X^i(t,0,x)} \rho_0(x)\di x  
    \\ & \leq \int_{\R^3} \abs{x} \rho_0(x)\di x  + \int_{\R^3} \int_0^T \abs{u^i(s, X^i(s,0,x))} \rho_0(x) \di s  \di x  
    \\ & \leq \int_{\R^3} \abs{x} \rho_0(x)\di x  + T \norm{u}_{L^\infty_t L^\infty_x} \norm{\rho_0}_{L^1_x} <+\infty.
\end{align}
To conclude, we remark that \eqref{finite first moment} holds true whenever the initial densities are compactly supported nonnegative $L^1$ functions. 
\end{rem}

\subsection{Outline of the strategy}

The proof of \cref{estimate on wasserstein} is based on the Lagrangian technique presented in \cite{CS21}. Fix a non decreasing function $\Theta\colon[1,3) \to [1, +\infty)$ and let $\rho_0 \in L^\Theta(\R^3)$. Take a Lagrangian solution $\rho$ to \eqref{transport-stokes} according to \cref{Lagrangian solution} with initial condition $\rho_0$. By \cref{norms are preserved}, we infer that $\rho\in L^\infty([0,T]; L^\Theta(\R^3))$. If $\omega_\Theta$ defined by \eqref{modulus of continuity} satisfies \eqref{convexity} and \eqref{osgood}, then it turns out that the corresponding velocity field $u$ has an Osgood modulus of continuity. Thus, by means of quantitative estimates, $u$ has a classically defined flow (see \cref{facts on classical RLF}) and, exploiting the machinery provided by \cite{CS21}, we are able to translate well posedness for the ODE system into well posedness for the nonlinear PDE \eqref{transport-stokes}. 

\begin{prop} \label{regularity of velocity}
Let $\Theta\colon [1, 3) \to [1, +\infty)$ be a non decreasing function. Letting $E$ the Oseen tensor \eqref{kernel}, given $\Phi \in L^\Theta$ then $E*\Phi \in L^\infty \cap C^{\omega_\Theta}$ and for any $x,y \in \R^3$ there holds that 
\begin{equation}
\int_{\R^3} \abs{ E(x-z) - E(y-z)} \abs{\Phi(z)} \di z \leq C \norm{\Phi}_{L^\Theta_x} \omega_{\Theta}(\abs{x-y}), \label{increment}
\end{equation}
\begin{equation}
    \int_{\R^3} \abs{ E(x-z) } \abs{\Phi(z)}\di z \leq C \norm{\Phi}_{L^\Theta_x}, \label{L^infty bound}
\end{equation}
\end{prop} 

\begin{rem} \label{the flow is classical}
Under the assumptions of \cref{regularity of velocity}, if we suppose in addition that $\omega_\Theta$ satisfies \eqref{osgood}, then the velocity field associated to a Lagrangian solution to \eqref{transport-stokes} with initial datum in $L^\Theta$ satisfies the assumptions of \cref{facts on classical RLF}. Thus, the corresponding flow is classically defined.  
\end{rem}

The proof of \cref{estimate on wasserstein} relies on an estimate for a suitable integral distance.  

\begin{prop} \label{estimate for the relative distance} 
Let $\Theta\colon [1, +\infty) \to [1, +\infty)$ be any non decreasing function. Let $\rho^1, \rho^2$ be Lagrangian solutions to \eqref{transport-stokes} according to \cref{Lagrangian solution} with nonnegative initial conditions $\rho_0^1, \rho_0^2$. Assume that \eqref{same norm}, \eqref{a priori growth} and \eqref{convexity} are satisfied. Suppose that $\norm{\rho_0^1}_{L^1_x}\neq 0$. Let $T_0\colon \R^3 \to \R^3$ be an optimal map between $\rho_0^1$ and $\rho_0^2$ according to \cref{optimal map}. Let $ X^1, X^2\colon[0,T]^2 \times \R^3 \to \R^3$ be the flows associated to $\rho^1, \rho^2$, respectively. We define the quantity 
\begin{equation}
Q(t) =  \frac{1}{\norm{\rho_0^1}_{L^1_x}} \int_{\R^3 \times \R^3} \abs{X^1(t,0, x)- X^2(t,0, T(x))} \rho_0^1(x) \di x  . \label{integral distance}
\end{equation} 
Then, for any $ t> 0$, there holds that 
\begin{equation} \label{dis_bella}
Q(t) \leq \W_1(\rho_0^1, \rho_0^2) + C \int_0^t \omega_{\Theta}(Q(s)) \di s, 
\end{equation} 
where the constant $C>0$ depends only on $ \norm{\rho_0^i}_{L^\Theta_x}$ for $i=1,2$. 
\end{prop}

With the help of \cref{estimate for the relative distance}, we are able to conclude the proof of \cref{estimate on wasserstein}. 

\begin{proof} [Proof of \cref{estimate on wasserstein}]
To begin, we notice that the case of $\rho_0^1 \equiv \rho_0^2 \equiv 0$ is trivial. Indeed, since any rearrangement invariant norm is preserved along the time evolution (see \cref{norms are preserved}), it is clear that for any $t \in [0,T]$ we have that $\rho^1(t,\cdot) = \rho^2(t,\cdot) =0$ for a.e.\ $x \in \R^3$, yielding \eqref{dis bella}. Assume that $\rho_0^1, \rho_0^2$ do not vanish identically. Thus, if we define $Q(t)$ as in \eqref{integral distance}, we get \eqref{dis_bella} with a uniform constant $C = C(\norm{\rho_0^1}_{L^\Theta_x}, \norm{\rho_0^2}_{L^\Theta_x})>0$. Hence, with the same argument of proof of \cref{RLF for osgood}, we find a function $\Tilde{\Gamma}\colon [0,T] \times [0, +\infty)^3 \to [0, +\infty)$ depending only on $T, \Theta$ with the following properties: 
\begin{itemize}
    \item $\Tilde{\Gamma}$ is non decreasing with respect to any variable; 
    \item $\Tilde{\Gamma} (\cdot, \sigma, \cdot, \cdot)\to 0$ uniformly on compact sets of $[0,T] \times [0, +\infty)^2$ as $\sigma \to 0$;
    \item there holds that 
    \begin{equation}
        Q(t) \leq \Tilde{\Gamma} (t, \W_1(\rho_0^1, \rho_0^2), \norm{\rho_0^1}_{L^\Theta_x}, \norm{\rho_0^2}_{L^\Theta_x}). \label{osgood quantitative}
    \end{equation}
    \end{itemize}   
To conclude, we check that 
\begin{equation}
\W_1(\rho^1(t), \rho^2(t)) \leq \norm{\rho_0^1}_{L^1_x} Q(t) \ \ \ \forall t \in [0,T], \label{upper bound wasserstein}
\end{equation}
yielding \eqref{dis bella} with 
\begin{equation}
    \Gamma(\tau, \sigma, \beta_1, \beta_2) = \Theta(1) \beta_1 \Tilde{\Gamma}(\tau, \sigma, \beta_1, \beta_2). \label{explicit gamma}
\end{equation}
By the push-forward formula, we can write 
\begin{equation}
Q(t) \norm{\rho_0^1}_{L^1_x} = \int_{\R^3} \abs{X^1(t,0,x) - X^2(t, 0,T(x))} \rho_0^1(x) \di x  = \int_{\R^3 \times \R^3} \abs{x-y} \di \mu_t(x,y), \nonumber
\end{equation} 
where $\mu_t$ is the Borel measure on $\R^3 \times \R^3$ defined by 
\begin{equation}
\mu_t = (X^1(t,0, \cdot) \times (X^2(t,0, T(\cdot)))) _\# (\rho_0^1 \mathcal{L}^3). \nonumber
\end{equation}
If we denote by $\pi_1, \pi_2$ the projections onto the first and the second component respectively, it is easy to compute the marginals of the measure $\mu_t$. Indeed, by \eqref{representation of solution} and \eqref{push forward}, we have that 
\begin{equation}
\begin{split}
(\pi_1)_\# \mu_t & = (X^1(t,0, \cdot))_\# (\rho_0^1 \mathcal{L}^3) = \rho^1(t) \mathcal{L}^3,    
\\ (\pi_2)_\# \mu_t & = (X^2(t,0, T(\cdot)))_\# (\rho_0^1 \mathcal{L}^3) = (X^2(t, 0, \cdot))_\# (\rho_0^2 \mathcal{L}^3) = \rho^2(t) \mathcal{L}^3.
\end{split}
\end{equation}
Hence, $\mu_t$ is a transport plan between $\rho^1(t)$ and $\rho^2(t)$ according to \cref{transport plan} (see \cref{maps vs plans}). Thus, \eqref{upper bound wasserstein} holds true by \cref{optimal map}. In conclusion, if $\rho_0^1 = \rho_0^2$, by \eqref{dis bella} and the properties of $\Gamma$, we have that 
\begin{equation}
\W_1(\rho^1(t), \rho^2(t)) =0 \ \ \ \forall t \in [0,T]. \nonumber 
\end{equation}
Since $\W_1$ is a distance, we infer that for any $t \in [0,T]$ we have that $\rho^1(t,\cdot)= \rho^2(t, \cdot)$ for a.e.\ $x \in \R^3$. 
\end{proof}

\subsection{Regularity of the velocity field}

The proof of \cref{regularity of velocity} is based on the following lemma, in the spirit of \cite{BBC16bis}*{Lemma 8.1}. For the reader's convenience, we give a detailed proof. We denote by $\tau_h$ the translation operator, namely for any function $\Phi\colon \R^d \to \R$ we set
$$\tau_h \Phi(x) = \Phi(x+h) \ \ \ h \in \R^d. $$

\begin{lem} \label{translations}
For any $p \in (3/2,3)$ and for any $h \in \R^3$ there holds that 
\begin{equation}
\norm{\tau_h E -E}_{L^p_x} \leq \frac{C}{(3-p)(2p-3)} \abs{h}^{\frac{3}{p} -1}, \label{estimate translations}
\end{equation}
where the constant $C>0$ is independent of $p$ and $h$. 
\end{lem}

\begin{proof}
Fix $h \in \R^3, h \neq 0$. Given $x \in \R^3$, we shall consider two cases. If $\abs{x} >2 \abs{h}$, then the segment joining $x$ and $x+h$ does not intersect the ball of radius $\abs{h}$ centered at the origin. Indeed, since $\abs{h} < \abs{x}/2$, for any $\theta \in [0,1]$ we have that 
$$ \abs{x+ \theta h} \geq \abs{x} - \abs{h} > \frac{\abs{x}}{2}. $$
Therefore, by Lagrange Theorem we estimate 
\begin{equation}
\abs{E(x+h)- E(x)} \leq C \abs{h} \sup_{\theta \in [0,1]} \abs{\nabla E(x+ \theta h)} \leq C \abs{h} \sup_{\theta \in [0,1]} \frac{1}{\abs{x+ \theta h}^2} \leq C \frac{\abs{h}}{\abs{x}^2}. \nonumber
\end{equation}
Then, given $p > 3/2$, we infer that 
\begin{align}
\int_{\abs{x} >2 \abs{h}} \abs{E(x+h) -E(x)}^p \di x  & \leq C_1^p \int_{\abs{x}>2 \abs{h}} \frac{\abs{h}^p}{\abs{x}^{2p}} \di x   = C_1^p \abs{h}^p \int_{2 \abs{h}}^{+\infty} r^{2-2p} \di r
\\ & \leq \frac{C_1^p}{2p -3} \abs{h}^{3-p}.  
\end{align}
If $\abs{x} \leq 2\abs{h}$, since the function $s \mapsto s^p$ is convex, we obtain that 
\begin{equation}
\abs{E(x+h) - E(x)}^p \leq 2^{p-1} \left[ \abs{E(x+h)}^p + \abs{E(x)}^p \right] \leq C_2^p \left[ \frac{1}{\abs{x+h}^{p}} + \frac{1}{\abs{x}^p}\right] . \nonumber
\end{equation}
Therefore, for $p < 3$ we have that 
\begin{align}
\int_{\abs{x} \leq 2 \abs{h}} \abs{E(x+h) - E(x)}^p \di x  & \leq C_2^p \int_{\abs{x} \leq 2 \abs{h}} \left[ \frac{1}{\abs{x+h}^{p}} + \frac{1}{\abs{x}^p}\right] \di x  \leq C_2^p \int_{\abs{x}\leq 3 \abs{h}} \frac{1}{\abs{x}^p} \di x 
\\ & \leq C_2^p \int_{0}^{3 \abs{h}} r^{2-p}\di r  \leq \frac{C_2^p}{3-p} \abs{h}^{3-p}.  
\end{align}
To summarize, for any $p \in (3/2, 3)$ we have that
\begin{align} 
\left[\int_{\R^3} \abs{E(x+h) -E(x)}^p \di x  \right]^{\frac{1}{p}} & \leq C_3 \abs{h}^{\frac{3}{p} -1} \left[ \frac{1}{2p-3} + \frac{1}{3-p} \right]^{\frac{1}{p}} \leq \frac{C_3}{(2p-3)(3-p)} \abs{h}^{\frac{3}{p} -1}. 
\end{align} 
\end{proof}

The proof of \cref{regularity of velocity} is based on the knowledge of the explicit constant in \eqref{estimate translations}. 

\begin{proof} [Proof of \cref{regularity of velocity}]
Fix $x,y \in \R^3$ and, since $\Phi \in L^\Theta$, by H\"older's inequality for $p \in (1,3)$ we have that 
\begin{equation}
\int_{\R^3} \abs{ E(x-z) - E(y-z)} \abs{\Phi(z)} \di z \leq \norm{\tau_{x-y} E -E}_{L^{p'}_x} \norm{\Phi}_{L^p_x} \leq \norm{\tau_{x-y} E -E}_{L^{p'}_x} \Theta(p), \label{regularity of velocity est 1}
\end{equation}
where $p'$ is the conjugate exponent of $p$. Since $p \in (3/2,3)$ if and only if $p'$ is in the same range, by means of \cref{translations} we estimate the right hand side of \eqref{regularity of velocity est 1} for $p \in (3/2,3)$:
\begin{align}
\int_{\R^3} \abs{ E(x-z) - E(y-z)} \abs{\Phi(z)} \di z & \leq \frac{C \Theta(p)}{(3-p')(2p'-3)} \abs{x-y}^{\frac{3}{p'} -1}  = \frac{C (p-1)^2 \Theta(p)}{(2p-3)(3-p)} \abs{x-y}^{2-\frac{3}{p}}.  
\end{align}
Since we are interested in the behaviour for $p \uparrow 3$, for any $p \in [2,3)$ we have that 
\begin{equation}
\int_{\R^3} \abs{ E(x-z) - E(y-z)} \abs{\Phi(z)} \di z \leq \frac{C \Theta(p)}{3-p} \abs{x-y}^{2- \frac{3}{p}}. \label{difference 1}
\end{equation}
For $x,y$ such that $\abs{x-y}\leq 1$, we can choose $p \in [2,3)$ such that 
$$\frac{1}{3-p} = 1- \log{\abs{x-y}}. $$ 
With this choice of $p$, we find that 
\begin{align}
\int_{\R^3} \abs{ E(x-z) - E(y-z)} \abs{\Phi(z)} \di z & \leq C (1-\log\abs{x-y}) \Theta\left( \frac{2-3 \log{\abs{x-y}}}{1- \log\abs{x-y}} \right) \abs{x-y}^{1 - \frac{1}{2-3\log \abs{x-y}}} 
\\ & \leq C \omega_\Theta(\abs{x-y}),
\end{align}
since the function 
$$r^{\frac{1}{3\log(r)-2}} = \exp \left( \frac{\log(r) }{3\log(r)-2} \right) $$
is uniformly bounded for $r \in (0, 1)$. This proves \eqref{increment} for $\abs{x-y}\leq 1$. For $x,y$ such that $\abs{x-y}\geq 1$, we can use the $L^\infty$ bound. Indeed, recalling that $E_1 \in L^2(\R^3)$ and $E_2 \in L^\infty(\R^3)$ (see \cref{summability of velocity}), for any $x \in \R^3$ we have that 
\begin{align} 
\norm{E*\Phi}
\abs{E} * \abs{\Phi}(x) & \leq \abs{E_1}* \abs{\Phi} (x) + \abs{E_2} * \abs{\Phi} (x) \nonumber
\\ & \leq \left[ \norm{E_1}_{L^\infty_x} \norm{\Phi}_{L^1_x}  + \norm{E_2}_{L^2_x} \norm{\Phi}_{L^2_x} \right] \leq C \norm{\Phi}_{L^\Theta_x}, \nonumber
\end{align}
thus proving \eqref{L^infty bound} and \eqref{increment} for $x,y \in \R^3$ such that $\abs{x-y} >1$. 
\end{proof}

\begin{rem} \label{continuity by CZ}
Notice that a similar estimate for the modulus of continuity \eqref{increment} can be achieved since the second derivative of the Oseen tensor $E$ is a Calderon--Zygmund operator (see \eqref{equibound 4}). Using the optimal constants in the Sobolev embedding and in the Calderon--Zygmund estimate, the interested reader can check that
\begin{equation}
    \abs{E*\Phi(x) -E*\Phi(y)} \leq C \omega_\Theta(\abs{x-y}). \nonumber
\end{equation}
However, in the proof of \cref{estimate for the relative distance}, the stronger inequality \eqref{increment} is needed. Moreover, our proof is completely elementary. 
\end{rem}

\subsection{Estimate on the Wasserstein distance} 

We discuss the proof of \cref{estimate for the relative distance}. 

\begin{proof} [Proof of \cref{estimate for the relative distance}]
By \eqref{L^1 norm of f}, we recall $\norm{\rho_i(t)}_{L^\Theta_x} = \norm{\rho_0^i}_{L^\Theta_x}$ for any $t \in [0,T]$ and $i=1,2$. Hence, for simplicity, we may assume that $\norm{\rho_0^i}_{L^1_x} = 1$. Letting $u^i= E*(- e_3 \rho^i)$ be the velocity field associated to $\rho^i$, by \eqref{ODE}, for a.e.\ $x,y \in \R^3$ and for any $t\in [0,T]$ we have that  
\begin{equation} \label{part 0}
    \begin{split}
        \abs{ X^1(t,0,x) - X^2(t,0,T(x)) } & \leq \abs{x-T(x)} + \int_0^t \abs{ u^1(s,X^1(s,0,x)) - u^2(s, X^2(s,0,T(x)))} \di s 
\\ & \leq \abs{x-T(x)} + \int_0^t \abs{ u^1(s,X^1(s,0,x)) - u^1(s, X^2(s,0,T(x)))} \di s 
\\ & \quad + \int_0^t \abs{u^1(s,X^2(s,0,T(x))) - u^2(s, X^2(s,0,T(x)))} \di s. 
    \end{split}
\end{equation}
Using \cref{regularity of velocity} we estimate the second term. Indeed, there exists a constant $C>0$ depending only on $\norm{\rho_0^1}_{L^\Theta_x}$ such that for a.e.\ $x,y \in \R^3$ and for any $t \in [0,T]$ we have 
\begin{equation}
\int_0^t \abs{u^1(s, X^1(s,0,x)) - u^1(s, X^2(s,0,T(x)))}\di s \leq C \int_0^t  \omega_\Theta(\abs{X^1(s,0,x) - X^2(s,0,T(x))}) \di s. \nonumber
\end{equation} 
Integrating with respect to the probability measure $\rho_0^1 \mathcal{L}^3$ and using Jensen's inequality (recall that $\omega_\Theta$ is concave), we have that 
\begin{equation} \label{part 1}
    \begin{split}
      \int_{\R^3} \int_0^t & \abs{u^1(s, X^1(s,0,x)) - u^1(s, X^2(s,0,T(x)))} \rho_0^1(x)\di s \di x  
    \\ &  \leq C \int_0^t \int_{\R^3} \omega_\Theta(\abs{X^1(s,0,x) - X^2(s,0,T(x))}) \rho_0^1(x) \di x \di s  
    \\ & \leq C \int_0^t \omega_\Theta \left( \int_{\R^3} \abs{X^1(s,0,x) - X^2(s,0,T(x))} \rho_0^1(x)\di x  \right) \di s = C \int_0^t Q(s) \di s.
    \end{split}
\end{equation}
To estimate the third term in \eqref{part 0}, notice that by \eqref{push forward} we have that 
\begin{align}
\int_{\R^3} & \abs{u^1(s,X^2(s,0,T(x))) - u^2(s, X^2(s,0,T(x)))} \rho_0^1(x) \di x  
\\ & = \int_{\R^3} \abs{u^1(s, X^2(s,0,x)) -u^2(s, X^2(s,0,x)) } \rho_0^2(x) \di x.
\end{align}
Thus, writing the convolution, using \eqref{representation of solution}, \eqref{push forward} and the push-forward formula, we find
\begin{align}
& \abs{ u^1(s, X^2(s,0,x))  - u^2(s,0, X^2(s,0,x))} 
\\ & \quad =  \abs{ E*(- e_3 \rho^1) (s, X^2(s,0,x)) - E*(- e_3 \rho^2) (s, X^2(s,0,x)) }
\\ & \quad = \bigg| \int_{\R^3}  E(X^2(s,0,x) - y) (- e_3 \rho^1(s,y) + e_3\rho^2(s,y) ) \di y \bigg|
\\ & \quad = \bigg| \int_{\R^3}  E(X^2(s,0,x) - X^1(s,0,y)) \rho_0^1(y)\di y - \int_{\R^3}  E(X^2(s,0,x) - X^2(s,0,y)) \rho_0^2(y) \di y \bigg|
\\ & \quad \leq \int_{\R^3} \abs{  E(X^2(s,0, x) - X^1(s,0,y)) -  E(X^2(s,0, x) - X^2(s,0,T(y))) } \rho_0^1(y) \di y.
\end{align} 
Therefore, by Fubini's Theorem and using again the push-forward formula, we infer that 
\begin{align}
\int_{\R^3} & \abs{u^1(s, X^2(s,0, x)) - u^2(s, X^2(s,0,x))} \rho_0^2(x) \di x  
\\ & \leq \int_{\R^3} \left[ \int_{\R^3}  \abs{ E(X^2(s,0,x) - X^1(s,0,y)) - E(X^2(s,0,x) - X^2(s,0,T(y)))} \rho_0^2(x) \di x  \right] \rho_0^1 (y) \di y 
\\ & = \int_{\R^3} \left[ \int_{\R^3}  \abs{ E(x - X^1(s,0,y)) - E(x - X^2(s,0,T(y)))} \, \rho^2(s,x) \right] \rho_0^1(y) \di y.
\end{align}
By \cref{regularity of velocity} and Jensen's inequality, we find a constant $C = C(\norm{\rho_0^2}_{L^\Theta_x})>0$ such that 
\begin{equation} \label{part 2} 
    \begin{split}
        \int_0^t \int_{\R^3} & \abs{u^1(s, X^2(s,0,T(x))) - u^2(s, X^2(s,0,T(x)))}  \rho_0^1(x) \di x  \di s 
\\ & \leq \int_0^t \int_{\R^3}  \int_{\R^3}  \abs{ E(x - X^1(s,0,y)) - E(x - X^2(s,0,T(y)))} \rho^2(s,x) \di x   \, \rho_0^1(y) \di y \di s 
\\ & \leq C  \int_0^t \int_{\R^3} \omega_\Theta(\abs{X^1(s,0,y) - X^2(s,0,T(y))}) \rho_0^1(y) \di y \di s 
\\ & \leq C \int_0^t \omega_\Theta \left(\int_{\R^3} \abs{X^1(s,0,y) - X^2(s,0,T(y))} \rho^1_0(y)\di y \right) \di s = C \int_0^t \omega_\Theta (Q(s)) \di s. 
    \end{split}
\end{equation}

Finally, combining \eqref{minimality}, \eqref{part 1} and \eqref{part 2}, we obtain \eqref{dis_bella}.
\end{proof}

\begin{rem}
The proofs of \cref{translations}, \cref{regularity of velocity} and \cref{estimate for the relative distance} hold true without any modification for general kernel $\Tilde{E}$ that is smooth out of the origin and such that 
\begin{equation}
\abs{\Tilde{E}(x)} \leq \frac{C}{\abs{x}} \ \ \ \abs{\nabla \Tilde{E}(x)} \leq \frac{C}{\abs{x}^2}. 
\end{equation}
We also remark that the strong inequality \eqref{increment} is really needed (see also \cref{continuity by CZ}).
\end{rem}

\section{Building explicit examples} \label{s:building explicit example}

Before discussing the proofs of \cref{symmetry regime} and \cref{explicit example}, we fix some notation.

\begin{Def} \label{cylinders and rotation}
Given $\theta \in \R$, let $R_\theta\colon \R^3 \to \R^3$ be the counter-clockwise rotation of angle $\theta$ along the axis $\Span(e_3)$, namely 
\begin{equation}
R_\theta(x) = \left( \begin{array}{cc} 
     x_1 \cos \theta - x_2\sin \theta
     \\ x_1 \sin\theta + x_2 \cos \theta 
     \\ x_3
\end{array}
\right) \ \ \ x\in \R^3 \nonumber .
\end{equation}
For any $\delta>0$, we denote by $\C_\delta$ the infinite cylinder of size $\delta$ around $\Span(e_3)$, namely
$$\C_{\delta} = \{(y,z) \in \R^2 \times \R \ \colon  \ \abs{y} \leq \delta \}. $$
\end{Def}

\subsection{Axisymmetric invariance} The proof of \cref{symmetry regime} relies on a direct computation. 

\begin{lem} \label{rotated velocity field}
Let $\rho \in L^1(\R^3)$ and let $u = E*(-e_3 \rho)$, where $E$ is the Oseen tensor \eqref{kernel}. Letting $\Tilde{\rho}(x) = \rho(R_\theta(x))$ and $\Tilde{u} = E*(-e_3 \Tilde{\rho} )$, we have that 
\begin{equation}
    \Tilde{u}(x) = R_{-\theta}( u( R_\theta(x))). \label{rotation of velocity}
\end{equation}
\end{lem}

\begin{proof} 
Given a $A \in \R^{3 \times 3}$, we denote by $[A]^3$ the third column of the matrix $A$. Thus, a simple computation shows that 
\begin{equation}
[R_\theta(x \otimes x)]^3 = \left( \begin{array}{cc} 
     x_1 x_3 \cos \theta - x_2 x_3\sin \theta
     \\ x_1 x_3 \sin\theta + x_2 x_3 \cos \theta 
     \\ x_3^2 
\end{array}
\right) = [R_\theta(x) \otimes R_\theta(x)]^3 \ \ \ x\in \R^3 \nonumber,
\end{equation}
yielding 
\begin{equation}
    [ E(R_\theta(x))]^3 = R_\theta([E(x)]^3). \nonumber
\end{equation}
Therefore, since $R_\theta$ is an isometry of $\R^3$, for a.e.\ $ x \in \R^3$, we have that 
\begin{align}
    \Tilde{u}(x) & = - \int_{\R^3} [E(x-y)]^3 \rho( R_\theta(y) )\di y  = -\int_{\R^3} [E(x- R_{-\theta}(z))]^3 \rho(z)\di z 
    \\ & = -\int_{\R^3} [E(R_{-\theta}( R_\theta(x)-z))]^3 \rho(z)\di z  = - R_{-\theta} \int_{\R^3} [E (R_\theta(x) -z)]^3 \rho(z)\di z 
    \\ & = R_{-\theta} ( u(R_\theta(x))). 
\end{align}
\end{proof}

\begin{prop} \label{rotated solution}
Let $\rho \in L^\infty([0,T];L^1(\R^3))$ be a Lagrangian solution to \eqref{transport-stokes} according to \cref{Lagrangian solution} with initial datum $\rho_0 \in L^1(\R^3)$. Given an angle $\theta$, we set $\Tilde{\rho}(t,x) = \rho (t, R_\theta(x))$. Then $\Tilde{\rho} \in L^\infty([0,T]; L^1(\R^3))$ is a Lagrangian solution to \eqref{transport-stokes} according to \cref{Lagrangian solution} with initial datum $\Tilde{\rho}_0(x) = R_\theta(\rho_0(x))$. 
\end{prop}

\begin{proof}
Since $R_\theta$ is an isometry of $\R^3$, we have that $\Tilde{\rho} \in L^\infty([0,T]; L^1(\R^3))$ and $\Tilde{\rho}_0 \in L^1(\R^3)$. Let $X\colon [0,T]^2 \times \R^3 \to \R^3$ be the regular Lagrangian flow generated by the velocity field $u = E*(-e_3 \rho)$. We claim that $\Tilde{X}(t,0,x) = R_{-\theta} ( X(t,0, R_\theta(x)))$ is the regular Lagrangian flow generated by $\tilde{u} = E*(-e_3 \Tilde{\rho})$ starting at time $0$. Indeed, $\Tilde{X}(t,0, \cdot)$ is a measure preserving map for any $t$ and, for a.e.\ $x \in \R^3$, by \cref{rotated velocity field} we have that 
\begin{equation}
\begin{cases}
\Tilde{X}(0,0,x) = x, 
\\ \displaystyle \frac{d}{dt} \Tilde{X}(t,0,x) = R_{-\theta} ( u (t, X(t,0, R_\theta(x))) ) =  \Tilde{u} (t, \Tilde{X}(t,0,x))
\end{cases} \nonumber
\end{equation}
in the integral sense. To conclude, since $\rho$ is a Lagrangian solution, we have that 
\begin{align}
    \Tilde{\rho}(t, \Tilde{X}(t,0,x)) & = \rho(t, R_\theta(\Tilde{X}(t,0,x))) = \rho(t, X(t,0, R_\theta(x)))  = \rho_0(R_\theta(x)) = \Tilde{\rho}_0(x). 
\end{align}
\end{proof}

Thus, the proof of the first part of \cref{symmetry regime} follows immediately.

\begin{proof}[Proof of \cref{symmetry regime} - first part] 
Since $\rho_0$ is invariant under $R_\theta$, by \cref{rotated solution} $\rho^\theta(t,x) = \rho(t, R_\theta(x))$ is a Lagrangian solution to \eqref{transport-stokes} with initial datum $\rho_0$. By \cref{estimate on wasserstein}, for any $t \in [0,T]$ we infer that $\rho^\theta(t, \cdot) = \rho(t, \cdot)$ for a.e.\ $x \in \R^3$. 
\end{proof}

In the following lemmas, we prove basic properties of the velocity field and the flow associated to a Lagrangian solution in a Yudovich-type space. 

\begin{lem} \label{vertical axis is invariant}
Under the assumptions of \cref{symmetry regime}, the vertical axis $\Span(e_3)$ is invariant under the flow $X\colon[0,T]^2\times \R^3 \to \R^3$ generated by $u = E*(-e_3 \rho)$.
\end{lem}

\begin{proof}
Since the flow of the velocity field $u$ is pointwise defined (see \cref{the flow is classical}), by the proof of \cref{rotated solution}, we infer that 
\begin{equation}
    X(t,s,x) = R_{-\theta}(X(t, s,R_\theta(x))) \ \ \ (t,s,x)\in [0,T]^2\times \R^3. \nonumber
\end{equation}
If we restrict to $x \in \Span(e_3)$, we easily conclude that $X(t,s,x)$ is fixed by the rotation $R_{-\theta}$, that is $X(t,s,x) \in \Span(e_3)$ for any $t,s \in [0,T]$.     
\end{proof}

\begin{rem} \label{explicit computation velocity}
In the framework of \cref{symmetry regime}, given $x \in \Span(e_3)$, we can explicitly compute $u(t,x)$ in coordinates. Indeed, we have that 
\begin{equation}
    u(t,x) = - \int_{\R^3} [E(y-x)]^3 \rho(t,y)\di y = - \left( 
    \begin{array}{cc}
         E_{1,3} * \rho(t,x) 
         \\ E_{2,3} * \rho (t,x)
         \\ E_{3,3}* \rho(t,x)
    \end{array}
    \right).
\end{equation} 
Dropping the variable $t$ and using cylindrical coordinates we have that 
\begin{align}
    E_{1,3}*\rho(x) + i E_{2,3}*\rho(x) & = \int_{\R^3} \frac{(y_1+ iy_2) y_3 }{\abs{y}^3} \rho(x-y)\di y  
    \\ & = \int_{\R} z\di z \int_0^{+\infty} \frac{r^2}{(r^2 + z^2)^{\frac{3}{2}}} \di r  \int_0^{2\pi} e^{i\alpha} \rho(x_3-z, r, \alpha) \di \alpha, 
\end{align}
where $i$ is the imaginary unit. Fix $z \in \R$ and $r \in (0, +\infty)$ and denote by 
\begin{equation}
    I = \int_0^{2\pi} e^{i\alpha} \rho(x_3-z, r, \alpha) \di \alpha. 
\end{equation}
Since $\rho(x_3-z, r, \alpha) = \rho(x_3-z,r, \alpha+\theta)$, we have that 
\begin{align}
    I & = \int_0^{2\pi} e^{i(\alpha+\theta)} \rho(x_3-z, r, \alpha+ \theta) \di \alpha = e^{i\theta} \int_0^{2\pi} e^{i \alpha} \rho(x_3-z, r, \alpha) \di \alpha = e^{i\theta} I. 
\end{align}
Since $\theta \neq 2k \pi$ for $k \in \Z$, we conclude that $I=0$, yielding $u_1(t,x) = u_2(t,x) = 0$ whenever $x \in \Span(e_3)$. Finally, we remark that 
\begin{align}
    E_{3,3}*\rho(x) & = \int_{\R^3} \frac{1}{\abs{x-y}} \left( 1 + \frac{(x_3-y_3)^2}{\abs{x-y}^2}\right) \rho(y)\di y. 
\end{align}
Hence, if $\rho$ is nonnegative, then $u(t,x)$ points downward for $t \in [0,T]$ and $x \in \Span(e_3)$. 
\end{rem}

\begin{lem} \label{l:uniform limit}
Under the assumptions of \cref{symmetry regime}, the velocity field $u = E*(-e_3 \rho)$ satisfies
\begin{equation}
    \lim_{M \to +\infty} \sup_{t \in [0,T]} \sup_{\abs{x}\geq M} \abs{u(t,x)} = 0. \label{uniform limit of velocity}
\end{equation}
\end{lem}

\begin{proof}
Fix $t \in [0,T]$ and $x\in \R^3$ such that $\abs{x}\geq 10$. By the decomposition of the Oseen tensor from \cref{summability of velocity}, we have that 
\begin{align}
    \abs{E* (-e_3 \rho)(t,x)} & \leq \int_{B_{\abs{x}/2}} \abs{E(x-y)} \abs{\rho(t,y)} \di y + \int_{B_{\abs{x}/2}^c} \abs{E(x-y)} \abs{\rho(t,y)} \di y
    \\ & \leq \norm{\rho(t)}_{L^1_x} \norm{E}_{L^\infty(B_{\abs{x}/2}^c)} + \norm{E_1}_{L^2_x}  \norm{\rho(t)}_{L^2(B_{\abs{x}/2}^c)} + \norm{E_2}_{L^\infty_x}  \norm{\rho(t)}_{L^1(B_{\abs{x}/2}^c)}.
\end{align}
Since the velocity field is bounded (see \cref{regularity of velocity}) and the corresponding flow is measure-preserving, for any $p \in [1, 3)$ we have that 
\begin{align}
    \norm{\rho(t)}_{L^p(B_{\abs{x}/2}^c)}^p = \int_{X(t, B_{\abs{x}/2}^c)^{-1}} \abs{\rho_0(x)}^p \di x  \leq \int_{B_{\abs{x}/2- T \norm{u}_{L^\infty_{t,x}}}} \abs{\rho_0(x)}^p \di x. 
\end{align}
Thus, we infer that 
\begin{equation}
    \abs{u(t,x)} \leq \norm{\rho_0}_{L^1_x} \norm{E}_{L^\infty(B_{\abs{x}/2}^c)} + \norm{E_1}_{L^2_x}  \norm{\rho_0}_{L^2(B_{\abs{x}/2- T \norm{u}_{L^\infty_{t,x}})}^c} + \norm{E_2}_{L^\infty_x}  \norm{\rho_0}_{L^1(B_{\abs{x}/2- T \norm{u}_{L^\infty_{t,x}}}^c)}. \nonumber 
\end{equation}
To conclude, given $\e>0$, there exists $M>0$ such that for $\abs{x}\geq M$ we have that 
$$ \norm{E}_{L^\infty(B_{\abs{x}/2}^c)} \leq \e, \ \ \  \norm{\rho_0}_{L^p(B_{\abs{x}/2- T \norm{u}_{L^\infty_{t,x}}}^c)} \leq \e \ \ \ p =1,2, $$
thus proving \eqref{uniform limit of velocity}. 
\end{proof}

\begin{lem} \label{invariance of cylinder}
Under the assumptions of \cref{symmetry regime}, for any $\delta>0$ there exists $\e>0$ such that for any $(y,z) \in \C_\e$ for any $t \in [0,T]$ there holds that $X(t,0,(y,z)) \in \C_\delta$. 
\end{lem}

\begin{proof}
Assume by contradiction that there exists $\delta_0>0$ such that for any $\e\in \N$ there exists $(y_\e, z_\e) \in \C_\e$ and $t_\e \in [0,T]$ such that $X(t_\e, 0, (y_\e, z_\e)) \notin \C_{\delta_0}$. Since $u(t,x)\to 0$ uniformly as $\abs{x}\to +\infty$ uniformly in time (see \cref{l:uniform limit}), we find $M>0$ such that $z_\e \in [-M,M]$ for any $\e>0$. Thus, up to subsequences, we may assume that $(y_\e, z_\e) \to (0, z_{0})$ as $\e \to 0$, for some $z_0 \in [-M,M]$. By the stability of the flow with respect to the supremum norm (see \cref{RLF for osgood} and \cref{facts on classical RLF}), we infer that $X(\cdot, 0, (y_\e, z_\e)) \to X(\cdot, 0, (0, z_0))$ uniformly in $[0,T]$. However, since $\Span(e_3)$ is invariant under the flow (see \cref{vertical axis is invariant}), we deduce that $X(t,0,(0,z)) \in \Span(e_3)$ for any $t \in [0,T]$ while $X(t_\e, 0, (y_\e, z_\e)) \notin \C_{\delta_0}$ for any $\e>0$.  
\end{proof}

Finally, we conclude the proof of \cref{symmetry regime}. 

\begin{proof} [Proof of \cref{symmetry regime} - conclusion]
Let $\delta>0$ and $p \in [1, +\infty]$ be as in the statement. Let $\e>0$ be given by \cref{invariance of cylinder} such that $X(t, 0, \C_\e) \subset \C_\delta$ for any $t \in [0,T]$. Then, we infer that $X^{-1}(t,0, \C_\delta^c) \subset \C_\e^c$ for any $t \in [0,T]$. By \eqref{representation of solution} and \cref{facts on classical RLF}, we deduce that 
\begin{equation}
    \rho(t, (y,z)) = \rho_0(X^{-1}(t, 0, (y,z))) \label{representation formula tris}
\end{equation}
for any $t \in [0,T]$ and for a.e.\ $(y,z) \in \R^3$. Since $X(t,0, \cdot)$ is measure preserving, using the push-forward formula, if $p< +\infty$ we deduce that 
$$\norm{\rho(t)}_{L^p(\C_\delta^c)} \leq \norm{\rho_0}_{L^p(\C_\e^c)}. $$
Indeed, the same conclusion holds true if $p= + \infty$ by \eqref{representation formula tris}. 
\end{proof}

\subsection{An explicit computation}

The proof of \cref{explicit example} relies on an direct computation.

\begin{proof} [Proof of \cref{explicit example}]
Setting $\Theta(s) = 1- \log(3-s)$ and computing $\omega_\Theta$ by \eqref{modulus of continuity}, we have 
\begin{equation}
    \omega_\Theta(s) = \begin{cases}
    s(1-\log(s)) (1+ \log(1-\log(s))) & s \in (0,1),
    \\ 1 & s \in [1, +\infty). \nonumber
    \end{cases}
\end{equation}
Then, it is trivial to check that $\omega_\Theta$ satisfies the Osgood property. Moreover, computing the second derivative of $\omega_\Theta$ in $(0,1)$, we infer that $\omega_\Theta$ is a convex function. Indeed, we have that 
\begin{equation}
    \omega_\Theta(s)'' = -\frac{1}{s} \left[ 1+ \log\left( \frac{1-\log(s)}{s^2} \right) \right] \leq 0 \ \ \ \text{for } s \in (0,1). \nonumber
\end{equation}
Let $\rho_0$ be defined by \eqref{rho_0}. Using polar coordinates, we obtain that 
\begin{align}
    \norm{\rho_0}_{L^p_x}^p & = 4\pi \int_{0}^{1/e} \frac{r^{2-p}}{\abs{\log(r)}^{\frac{p}{3}}} \di r = 4 \pi \int_1^{+\infty} e^{-t(3-p)} t^{-\frac{p}{3}} \di t = 4\pi (3-p)^{-\frac{3-p}{3}} \int_{3-p}^{+\infty} e^{-z} z^{-\frac{p}{3}} \di z,
\end{align}
after the change of variables $r= e^{-t}$ and then $z=t(3-p)$. Notice that 
\begin{equation}
    4 \pi (3-p)^{-\frac{3-p}{3}} = 4 \pi \exp\left( -\frac{3-p}{3} \log(3-p) \right) \leq C \ \ \ \text{for } p \in [1,3), \nonumber
\end{equation}
since the function $q\log(q)$ is bounded around $0$. Then, we split the integral: 
\begin{align}
\int_{3-p}^{+\infty} e^{-z} z^{-p/3}\di z & \leq \int_{1}^{+\infty} e^{-z} \di z + \int_{3-p}^1 z^{-p/3} \di z  \leq 1 + \frac{1}{1-p/3} \left[ z^{1-p/3} \right]_{z=3-p}^{z=1} 
\\ & = 1 + \frac{3}{3-p} \left[ 1- (3-p)^{\frac{3-p}{3}} \right] = 1 + 3 \left[ \frac{1-\exp\left(\frac{3-p}{3}\log(3-p)\right)}{3-p}\right] 
\\ & \leq 1 -\log(3-p). 
\end{align}
Therefore, we conclude that 
\begin{equation}
    \norm{\rho_0}_{L^p_x} \leq C (1-\log(3-p)) \ \ \ \text{for } p \in [1, 3) \nonumber. 
\end{equation}
On the other hand, $\rho_0 \notin L^3$. Indeed, we have that 
\begin{equation}
    \int_0^{1/e} \frac{1}{\abs{x}^3 \abs{\log{x}}}\di x  = 4\pi \int_0^{1/e} \frac{1}{r \abs{\log(r)}} \ dr = +\infty. \nonumber
\end{equation}
Let $\rho$ be any Lagrangian solution to \eqref{transport-stokes} with initial condition $\rho_0$. By \eqref{L^1 norm of f}, we deduce that $\rho \in L^\infty([0,T];L^\Theta(\R^3))$, where $\Theta(s) =  1- \log(3-s)$ and $\rho(t)\notin L^3$ for a.e.\ $t \in [0,T]$.  

Let $\rho_1$ be defined by \eqref{tilde rho_0}. With the very same argument as before, it is easy to check that 
\begin{equation}
    \norm{\rho_1}_{L^p(\R^2)} \leq C (1-\log(3-p)) \ \ \ \text{for } p \in [1, 3) \nonumber,
\end{equation}
for a suitable explicit constant $C>0$. Then, for any function $\rho_2 \in L^1 \cap L^3(\R)$, we infer that $\Tilde{\rho}_0(x) = \rho_1(x_1,x_2) \rho_2(x_3)$ satisfies 
\begin{equation}
    \norm{\Tilde{\rho}_0}_{L^p(\R^3)} \leq C (1-\log(3-p)) \ \ \ \text{for } p \in [1, 3) \nonumber,
\end{equation}
where the constant $C>0$ depends also on $\rho_2$. Moreover, $\Tilde{\rho}_0 \notin L^3(\R^3)$, since $\rho_1 \notin L^3(\R^2)$. As before, we deduce that the corresponding Lagrangian solution $\rho$ with initial datum $\Tilde{\rho}_0$ lies in $ L^\infty([0,T]; L^\Theta(\R^3))$, but $\rho(t) \notin L^3$ for any $t \in [0,T]$. 
\end{proof}

\appendix

\section{Quantitative estimates for the regular Lagrangian flow} \label{s:stability estimate}

Apriori estimates play a pivotal role within the theory of regular Lagrangian flows. In this section, we show the basic results needed to prove \cref{quantitative RLF bis}, following \cites{BC13, CDL08}. We present a quantitative estimate for the regular Lagrangian flow associated to a vector field in $L^\infty(W^{1,\infty})+L^\infty(W^{1,p})$ for some $p \in (1,+\infty)$, such as the velocity field arising in \eqref{transport-stokes}. We achieve this estimate by adding a bounded Lipschitz perturbation to the Sobolev case studied in \cite{CDL08}.

\begin{thm} \label{quantitative RLF basic} 
Fix $p\in (1, +\infty)$ and let $b^1, b^2$ be vector fields with the following properties: 
\begin{itemize}
\item $b^i= b^i_1 + b^i_2$, where $b^i_1 \in L^1([0,T];L^1(\R^d))$ and $b^i_2 \in L^1([0,T]; L^\infty(\R^d))$ for $i=i,2$; 
\item $Db^1_1 \in L^1([0,T]; L^p(\R^d))$, $Db^1_2 \in L^1([0,T]; L^\infty(\R^d))$; 
\item $b^1, b^2$ are incompressible vector fields. 
\end{itemize}
Let $X^1, X^2$ be regular Lagrangian flows associated to $b^1, b^2$ respectively according to \cref{regular lagrangian flow}. Given $\delta>0, R>0$, we define the quantity
\begin{equation}
g_{\delta,R}(t,s) = \int_{B_R} \log\left( \frac{\abs{X^1(t,s,x) - X^2(t,s,x)}}{\delta} + 1 \right) \di x , \quad (t,s) \in [0,T]^2. \label{g_R,delta}
\end{equation}
Then, for any $R,\delta>0$ and for any $\lambda > R+\norm{b_2^2}_{L^1_t L^\infty_x }$ there holds that 
\begin{equation} \label{a priori est RLF basic}
    \begin{split}
        \sup_{t,s \in [0,T] }g_{\delta,R}(t,s) & \leq c_{d,p,R} \norm{Db^1_1}_{L^1_t L^p_x } + c_{d,R} \norm{Db_2^1}_{L^1_t L^\infty_x} +  \frac{1}{\delta} \norm{b^1_1 - b_1^2}_{L^1_t L^1_x} 
\\ & \quad + \frac{1}{\delta}\norm{b^1_2 -b^2_2}_{L^1_t L^1(B_\lambda)} + \frac{1}{\delta} \frac{\norm{b^1_2}_{L^1_t L^\infty_x} + \norm{b^2_2}_{L^1_t L^\infty_x}}{\lambda-R-\norm{b_2^2}_{L^1_t L^\infty_x}} \norm{b_2^1}_{L^1_t L^1_x} . 
    \end{split}
\end{equation}
\end{thm}

\begin{rem}
\cref{quantitative RLF basic} holds without any regularity assumptions on the vector field $b^2$.  
\end{rem}

We refer to \cref{ss:proof of the apriori est} for a detailed proof of \cref{quantitative RLF basic}. In particular, we  obtain the following result. 

\begin{prop} \label{quantitative RLF} 
Let $b^1, b^2$ be vector fields as in \cref{quantitative RLF basic} and let $X^1, X^2$ be regular Lagrangian flows for $b^1, b^2$ respectively according to \cref{regular lagrangian flow}. Assume that
\begin{equation}
    \max \left\{ \norm{b^1_{1}}_{L^1_t L^1_x}, \norm{b^1_2}_{L^1_t L^\infty_x}, \norm{b^2_{1}}_{L^1_t L^1_x}, \norm{b^2_2}_{L^1_t L^\infty_x}, \norm{Db^1_1}_{L^1_t L^p_x}, \norm{D b^1_2}_{L^1_t L^\infty_x}, \right\} \leq \overline{K} < +\infty. \label{equibound of norms}
\end{equation}
For any $\e, R, \eta>0$ there exist a constant $C>0$ and radius $\lambda >0$ such that 
\begin{equation} \label{a priori est RLF}
    \begin{split}
        \sup_{t,s \in [0,T]} & \meas^d(\{ x \in B_R \ \colon  \ \abs{ X^1(t,s,x) - X^2(t,s,x)}> \e \}) 
    \\ & \quad \leq C \left[ \norm{b^1_1-b^2_1}_{L^1_t L^1_x} + \norm{b^1_2-b^2_2}_{L^1_t L^1( B_\lambda )} \right] + \eta. 
    \end{split}
\end{equation} 
The constant $C$ and the radius $\lambda$ depend only on $\eta, \e, p, d, R, \overline{K}$. 
\end{prop}

\begin{proof}
Fix $R>0, \e>0$. For any $\delta>0, \lambda > R+\overline{K}$, by \eqref{a priori est RLF basic} we have that 
\begin{align}
    \sup_{t,s \in [0,T]} & \meas^d( \{ x \in B_R \ \colon  \ \abs{X^1(t,s,x) - X^2(t,s,x)} > \e \} ) 
    \\ & \leq \left[\log \left( \frac{\e}{\delta} +1 \right)\right]^{-1} \sup_{t,s \in [0,T]} g_{\delta,R}(t,s) \leq \left[\log \left( \frac{\e}{\delta} +1 \right)\right]^{-1} c_{d,R,p} \overline{K} \nonumber
    \\ & \quad + \left[\log \left( \frac{\e}{\delta} +1 \right) \delta \right]^{-1} \left( \norm{b^1_1 - b_1^2}_{L^1_t L^1_x} + \norm{b^1_2-b^2_2}_{L^1_t L^1(B_\lambda)} + \frac{2 \overline{K}^2}{\lambda-R-\overline{K}} \right) 
\end{align}
Hence, we fix $\eta>0$ and we choose $\delta = \delta(\e, \eta, R,p, d, \overline{K})>0$ such that 
$$\left[\log \left( \frac{\e}{\delta} +1 \right)\right]^{-1} c_{d,R,p} \overline{K} \leq \frac{\eta}{2}. $$
Once $\delta$ is fixed, we choose $\lambda = \lambda(\e, \eta, \delta, R, d,p, \overline{K} ) > R +  \overline{K}$ such that
$$ \left[ \log \left( \frac{\e}{\delta} +1 \right) \delta \right]^{-1} \frac{2 \overline{K}^2}{\lambda - R - \overline{K}} \leq \frac{\eta}{2}, $$
yielding \eqref{a priori est RLF}
\end{proof}

As a byproduct of \cref{quantitative RLF}, by an approximation procedure, we get uniqueness \cite{BC13}*{Theorem 6.1}, stability \cite{BC13}*{Theorem 6.2} and existence \cite{BC13}*{Theorem 6.4} of the regular Lagrangian flow associated to a vector field with the same properties as $b^1$ in \cref{quantitative RLF}, as well as the semigroup property of the (complete) regular Lagrangian flow \cite{BC13}*{Corollary 6.6}. We left the straightforward modifications to the interested reader. It is enough to use the stability estimate of \cref{quantitative RLF} in place of \cite{BC13}*{Proposition 5.9}. 

\begin{thm} \label{well posedness of the flow}
Let $b = b_1+b_2\colon [0,T] \times \R^d \to \R^d$ be an incompressible vector field with the following properties: 
\begin{itemize}
    \item $b_1 \in L^1([0,T]; L^1(\R^d))$, $b_2 \in L^1([0,T]; L^\infty(\R^d))$; 
    \item $Db_1 \in L^1([0,T]; L^p(\R^d))$ for some $p \in (1, \infty)$, $Db_2 \in L^1([0,T]; L^\infty(\R^d))$.
\end{itemize}
Then, there exists a unique regular Lagrangian flow $X\colon [0,T]^2 \times \R^d \to \R^d$ associated to $b$ according to \cref{regular lagrangian flow}. The flow map $X$ is in $C([0,T]^2, L^0_\loc(\R^d))$ and the semigroup property holds true, i.e.\ for any $t,\tau,s \in [0,T]$ we have that 
\begin{equation}
    X(t,\tau, X(\tau,s, x)) = X(t,s,x) \ \ \ \text{ for a.e.\ } x\in \R^d. \label{semigroup}
\end{equation}
\end{thm}

\cref{quantitative RLF bis} is a corollary of \cref{quantitative RLF} and \cref{well posedness of the flow}. 

\begin{proof}[Proof of \cref{quantitative RLF bis}]
Setting $u^i_j = E_j * (- e_3 \rho^i)$ for $i,j =1,2$, where $E= E_1+E_2$ as in \cref{summability of velocity}, we fix $p \in (1,3/2)$ and by \eqref{equibound} we have that 
\begin{equation}
    \max \left\{ \norm{u^1_{1}}_{L^1_t L^1_x}, \norm{u^1_2}_{L^1_t L^\infty_x}, \norm{u^2_{1}}_{L^1_t L^1_x}, \norm{u^2_2}_{L^1_t L^\infty_x}, \norm{Du^1_1}_{L^1_t L^p_x}, \norm{D u^1_2}_{L^1_t L^\infty_x}, \right\} \leq C_p T K, \nonumber
\end{equation}
yielding \eqref{equibound of norms}. Thus, by \eqref{a priori est RLF}, we obtain \eqref{stability of RLF}.
\end{proof}

\subsection{Proof of the basic apriori estimate} \label{ss:proof of the apriori est}

In this section, we give a detailed proof of \cref{quantitative RLF basic}, mainly following the technique of \cites{CDL08, BC13}. Given a function $f \in L^1_\loc(\R^d)$, we denote the (global) maximal function of $f$ by 
\begin{equation}
M f (x) = \sup_{0 < r < +\infty} \fint_{B_r(x)} \abs{f(y)} \di y. 
\end{equation}
The following estimate is available for the global maximal function, for any $p \in (1, +\infty]$ \cite{S70}:
\begin{equation}
\norm{M f}_{L^p_x} \leq c(d,p) \norm{f}_{L^p_x}. \label{bound global maximal function}
\end{equation}
The inequality above does not hold for $p = 1$. The maximal function of the derivative of a weakly differentiable vector field plays a crucial role in the estimate of difference quotients associated to a Sobolev vector fields. We recall the following result (see \cite{S70}), generalizing to the case of BV functions the trivial estimate for difference quotient for Lipschitz functions. 

\begin{thm} \label{diff quotient}
Given $u \in BV(\R^d)$, there exists a negligible set $\mathcal{N} \subset \R^d$ and a constant $c_d >0$ such that 
\begin{equation}
    \abs{u(x) -u(y)} \leq c_d \abs{x-y} (M(Db)(x) + M (Db)(y)) \ \ \ \forall x,y \in \R^d \setminus \mathcal{N}. \nonumber
\end{equation}
\end{thm} 

Since we aim to deal with unbounded vector fields, we need an estimate of the sublevel and the superlevel sets of the flow map. 

\begin{lem} \label{superlevel sets}
Let $b = b_1 + b_2$ be an incompressible vector field such that $b_1 \in L^1([0,T]; L^1(\R^d))$ and $b_2 \in L^1([0,T]; L^\infty(\R^d))$. Let $X\colon [0,T]^2 \times \R^d \to \R^d$ be a regular Lagrangian flow associated to $b$ according to \cref{regular lagrangian flow}. Then, for any $R, \lambda$ such that $\lambda > \norm{b_2}_{L^1_t L^\infty_x} + R >0$, there holds that 
\begin{equation}
\sup_{(t,s) \in [0,T]} \meas^d( \{ x \in B_R \ \colon  \ \abs{X(t,s,x)} \geq \lambda\} ) \leq \frac{\norm{b_1}_{L^1_t L^1_x}}{\lambda-R - \norm{b_2}_{L^1_t L^\infty_x}}, \label{superlevel 2}
\end{equation}
\begin{equation}
\sup_{(t,s) \in [0,T]} \meas^d( \{ x \in \R^d \setminus B_R \ \colon  \ \abs{X(t,s,x)} \leq \lambda\} ) \leq \frac{\norm{b_1}_{L^1_t L^1_x}}{\lambda -R- \norm{b_2}_{L^1_t L^\infty_x}}. \label{superlevel 1}
\end{equation}
\end{lem}

\begin{proof}
Fix $(t,s) \in [0,T]$. By \eqref{ODE}, for a.e.\ $\abs{x} \leq R$, we estimate 
$$ \abs{X(t,s,x)} \leq R + \int_0^T \abs{b_1(z, X(z,s,x))}\di z + \int_0^T \norm{b_2(z)}_{L^\infty_x} \di z, $$
yielding
\begin{equation}
\begin{split}
    \meas^d( \{ x \in B_R \ \colon  & \ \abs{X(t,s,x)} \geq \lambda\} )  
\\ & \leq \meas^d \left( \left\{ x \in \R^d \ \bigg| \ \int_s^{t} \abs{b_1(z, X(z,s,x))}\di z \geq \lambda - R - \norm{b_2}_{L^1_t L^\infty_x} \right\} \right) 
\\ & \leq \frac{1}{\lambda -R - \norm{b_2}_{L^1_t L^\infty_x}} \int_{\R^d} \int_s^T \abs{b_1(z, X(z,s,x))}\di z \di x  
\\ & \leq \frac{\norm{b_1}_{L^1_t L^1_x}}{\lambda -R - \norm{b_2}_{L^1_t L^\infty_x}} .
\end{split}
\end{equation}
The proof of \eqref{superlevel 1} is analogous. 
\end{proof}

Finally, we discuss the proof of \cref{quantitative RLF basic}. 

\begin{proof}
Fix $s \in [0, T]$. Since the map $t \mapsto X^i(t,s,x)$ is absolutely continuous for a.e.\ $x$ (for $i=1,2$), then we claim that $t \mapsto g_{\delta, R}(t,s)$ is absolutely continuous in $[0,T]$. Indeed, estimating the difference quotient of the Sobolev part of the vector field with \ref{diff quotient}, we have that
\begin{align}
& \abs{\partial_t g_{\delta,R}(t,s) }  \leq \int_{B_R} \frac{\abs{b^1(t, X^1(t,s,x))- b^2(t, X^2(t,s,x))}}{\delta + \abs{X^1(t,s,x) - X^2(t,s,x)}}\di x  
\\ & \leq \int_{B_R} \left[ \frac{\abs{b^1_1(t, X^1(t,s,x))- b^1_1(t, X^2(t,s,x))}}{\delta + \abs{X^1(t,s,x) - X^2(t,s,x)}} + \frac{\abs{b^1_1(t, X^2(t,s,x))- b^2_1(t, X^2(t,s,x))}}{\delta + \abs{X^1(t,s,x) - X^2(t,s,x)}} \right] \di x  
\\ & \quad + \int_{B_R} \left[ \frac{\abs{b^1_2(t, X^1(t,s,x))- b^1_2(t, X^2(t,s,x))}}{\delta + \abs{X^1(t,s,x) - X^2(t,s,x)}} + \frac{\abs{b^1_2(t, X^2(t,s,x))- b^2_2(t, X^2(t,s,x))}}{\delta + \abs{X^1(t,s,x) - X^2(t,s,x)}} \right] \di x  
\\ & \leq \int_{B_R} [M(Db^1_1(t)) (X^1(t,s,x)) + M(D b^1_1(t)) (X^2(t,s,x)) ] \di x  
\\ & \quad + \frac{1}{\delta} \int_{\R^d} \abs{b^1_1(t, X^2(t,s,x)) - b^2_1(t, X^2(t,s,x)) } \di x  
\\ & \quad + \omega_d R^d \norm{Db_2^1(t)}_{L^\infty_x} + \frac{1}{\delta} \int_{B_R} \abs{b^1_2(t, X^2(t,s,x)) - b^2_2(t, X^2(t,s,x)) } \di x.   
\end{align}
By H\"older's inequality, the push-forward formula (the flow map is measure preserving for any $t$) and \eqref{bound global maximal function}, we estimate the first term by 
\begin{align}
    \int_{B_R} & [M(Db^1_1(t)) (X^1(t,s,x)) + M(D b^1_1(t)) (X^2(t,s,x)) ] \di x 
    \\ & \leq c_{d,p,R} \left[ \left(\int_{\R^d} [M(Db^1_1(t)) (X^1(t,s,x)) ]^p \di x  \right)^\frac{1}{p} + \left(\int_{\R^d} [M(Db^1_1(t)) (X^2(t,s,x)) ]^p \di x  \right)^{\frac{1}{p}} \right] 
    \\ & \leq c_{d,p,R} \norm{Db^1_1(t)}_{L^p_x}. 
\end{align}
By the push-forward formula, we estimate the second term by 
\begin{equation}
    \int_{\R^d} \abs{b^1_1(t, X^2(t,s,x)) - b^2_1(t, X^2(t,s,x)) } \di x  = \norm{b^1_1(t)- b^2_1(t)}_{L^1_x}. 
\end{equation}
To estimate the last term, fix $\lambda>0$, $t,s \in [0,T]$ and we denote by  
$$G_{R,\lambda} = \{ x \in B_R \ \colon  \ \abs{X^2(t,s,x)} \leq \lambda \} . $$
Then, we estimate
\begin{align}
\int_{B_R} & \abs{b^1_2(t, X^2(t,s,x)) - b^2_2(t, X^2(t,s,x)) } \di x  = \int_{G_{R,\lambda}} \abs{b^1_2(t, X^2(t,s,x)) - b^2_2(t, X^2(t,s,x)) } \di x  
\\ & \quad +  \int_{B_R \cap G^c_{R, \lambda}} \abs{b^1_2(t, X^2(t,s,x)) - b^2_2(t, X^2(t,s,x)) } \di x 
\\ & \leq \int_{B_\lambda} \abs{b^1_2(t,x) - b^2_2(t,x)} \di x  + \meas^d(B_R \cap G^c_{R,\lambda}) \left[ \norm{b^1_2(t)}_{L^\infty_x} + \norm{b^2_2(t)}_{L^\infty_x}\right]. 
\end{align}
If we take $\lambda \geq R + \norm{b_2^2}_{L^1_t L^\infty_x}$, by \eqref{superlevel 2}, we obtain that 
\begin{align}
\int_{B_R} & \abs{b^1_2(t, X^2(t,s,x)) - b^2_2(t, X^2(t,s,x)) } \di x  
\\ & \leq \norm{b^1_2(t) -b^2_2(t)}_{L^1( B_\lambda)} + \frac{\norm{b^1_2(t)}_{L^\infty_x} + \norm{b^2_2(t)}_{L^\infty_x}}{\lambda-R-\norm{b_2^2}_{L^1_t L^\infty_x}} \norm{b^2_1}_{L^1_t L^1_x}
\end{align}
Therefore, the function $g_{\delta, R}(\cdot, s)$ is absolutely continuous with respect to the $t$ variable. Recall that $g_{\delta, R} (s,s) = 0$ by the property of the regular Lagrangian flow. To summarize, for any $(t,s)\in [0,T]$, for any $R,\delta>0$ and for any $\lambda > R+\norm{b_2^2}_{L^1_t L^\infty_x}$ we have that
\begin{align}
g_{\delta,R}(t,s) & \leq g_{\delta,R}(s,s) + \int_s^t \abs{\partial_z g_{\delta,R}(z,s)}\di z  \leq c_{d,p,R} \norm{Db^1_1}_{L^1_t L^p_x} +  c_{d,R} \norm{Db_2^1}_{L^1_t L^\infty_x } 
\\ & \quad + \frac{1}{\delta} \left[ \norm{b^1_1 - b_1^2}_{L^1_t L^1_x} + \norm{b^1_2 -b^2_2}_{L^1_t L^1(B_\lambda)} + \frac{\norm{b^1_2}_{L^1_t L^\infty_x} + \norm{b^2_2}_{L^1_t L^\infty_x}}{\lambda-R-\norm{b_2^2}_{L^1_t L^\infty_x}} \norm{b_2^1}_{L^1_t L^1_x} \right] . 
\end{align} 
\end{proof}

\section{The flow associated to an Osgood vector field} \label{s:classical flow}

In this section, we recall some basic result concerning the classical flow associated to a bounded vector field with modulus of continuity that satisfies the Osgood condition \eqref{osgood}. The following result is an apriori estimate in the same spirit of \cref{quantitative RLF basic}. The main tool is the following nonlinear Grownwall estimate, also known as Bihari--LaSalle lemma (see \cite{P98}). For the reader's convenience, we sketch a proof. 

\begin{lem} \label{RLF for osgood}
    Let $b^1,b^2 \colon [0,T] \times \R^d \to \R^d$ be bounded vector fields such that
    \begin{equation}
        \abs{b^i(t,x) - b^i(t,y)} \leq \omega(\abs{x-y}) \ \ \ \forall t \in [0,T], \ x,y \in \R^d, \ i=1,2,\label{osgood continuity}
    \end{equation}
    where $\omega$ is a modulus of continuity that satisfies the Osgood condition \eqref{osgood}. Let $\gamma_1, \gamma_2\colon [0,T] \to \R^d$ be integral solutions to the Cauchy problem 
    \begin{equation}
    \begin{cases}
        \gamma'_i(t) = b^i(t, \gamma_i(t)) & t \in [0,T], 
        \\ \gamma_i(s_i) = x_i,
    \end{cases} 
    \end{equation}
    for $i=1,2$. Then, there exists a function $\Gamma\colon [0,T]^2 \times [0, +\infty)^3 \to [0,+\infty)$ depending only on $\omega, T$ with the following properties: 
    \begin{itemize}
    \item $\Gamma$ is non decreasing with respect to any variable; 
    \item $\Gamma (\tau, \sigma, \rho, \beta, \cdot)\to 0$ uniformly on compact sets of $[0, +\infty)$ as $\abs{\tau} + \abs{\sigma}+ \abs{\rho} + \abs{\beta} \to 0$;
    \item there holds that 
    \begin{equation}
        \abs{\gamma_1(t_1) - \gamma_2(t_2)} \leq \Gamma (\abs{t_1-t_2}, \abs{s_1-s_2}, \abs{x_1-x_2}, \norm{b^1-b^2}_{L^1_t L^\infty_x}, \norm{b^2}_{L^\infty_t L^\infty_x}). 
    \end{equation}
    \end{itemize}   
\end{lem}

\begin{proof}
By the integral formulation we have that 
\begin{align}
   \abs{\gamma_1(t_1) - \gamma_2(t_1)} & \leq \abs{x_1-x_2} + \bigg| \int_{s_1}^{t_1} b^1(z, \gamma_1(z)) \di z - \int_{s_2}^{t_1} b^2(z, \gamma_2(z)) \di z \bigg|
   \\ & \leq \abs{x_1-x_2} + \int_0^{t_1} \abs{b^1(z, \gamma_1(z)) - b^2(z, \gamma_2(z))} \di z + \abs{s_1-s_2} \norm{b^2}_{L^\infty_t L^\infty_x}
   \\ & \leq \abs{x_1-x_2} + \int_0^{t_1} \omega(\abs{\gamma_1(z) -\gamma_2(z)}) \di z 
   \\ & \quad + \norm{b^1-b^2}_{L^1_t L^\infty_x} + \abs{s_1-s_2} \norm{b^2}_{L^\infty_t L^\infty_x}.
\end{align}
Given $\eta>0$, if we denote by 
$$g_\eta(t_1) = \abs{x_1-x_2} + \int_0^{t_1} \omega(\abs{\gamma_1(z)-\gamma_2(z)}) \di z + \abs{s_1-s_2} \norm{b^2}_{L^\infty_t L^\infty_x} + \eta, $$
we infer that $t_1 \mapsto g_\eta(t_1)$ is a Lipschitz function in $[0,T]$ satisfying
$$g_\eta'(t_1) = \omega(\abs{\gamma_1(t_1) -\gamma_2(t_1)}) \leq \omega(g_\eta(t_1)), $$
since $\omega$ is non decreasing. Moreover, if we set 
$$\Omega(z) = -\int_z^1 \frac{1}{\omega(s)} \di s $$
and recalling that $g_\eta(t_1) \geq \eta >0$ in $[0,T]$, by the chain rule for Sobolev functions we have that 
$$ \frac{d}{dt_1} \Omega(g_\eta(t_1)) = \frac{g'_\eta(t_1)}{\omega(g_\eta(t_1))} \leq 1 \ \ \ \text{for a.e.\ } t_1 \in [0,T].  $$
Since $\Omega\colon (0, +\infty) \to \R$ is bijective under \eqref{osgood}, after integrating with respect to $t_1$, we have that 
\begin{equation}
    g_\eta(t_1) \leq \Omega^{-1}( \Omega(g_\eta(0)) + t_1), \nonumber
\end{equation}
yielding 
\begin{equation} \label{explicit gamma appendix}
    \begin{split}
        \abs{\gamma_1(t_1) -\gamma_2(t_2)} & \leq \abs{\gamma_1(t_1)- \gamma_2(t_1)} + \abs{\gamma_2(t_1) + \gamma_2(t_2)} 
    \\ & \leq  \Omega^{-1}(\Omega( \abs{x_1-x_2} + \abs{s_1-s_2} \norm{b^2}_{L^\infty_t L^\infty_x} + \norm{b^1-b^2}_{L^1_t L^\infty_x} + \eta ) + T) 
    \\ & \quad + \abs{t_1-t_2} \norm{b^2}_{L^\infty_t L^\infty_x}. 
    \end{split}
\end{equation}
Letting $\eta \to 0$, by the Osgood property \eqref{osgood} it is easy to see that the right hand side of \eqref{explicit gamma appendix} defines a function $\Gamma\colon[0,T]^2 \times [0,+\infty)^3 \to [0, +\infty) $ with the required properties. 
\end{proof}

By a regularization procedure, \cref{RLF for osgood} gives existence, uniqueness and stability of the classical flow associated to a bounded vector field with Osgood modulus of continuity, as well as a Liouville-type theorem in the incompressible case. We leave the details to the reader.

\begin{thm} \label{facts on classical RLF}
    Let $b\colon [0,T] \times \R^d \to \R^d$ be a bounded vector field such that \eqref{osgood continuity} holds true, where $\omega$ satisfies the Osgood condition \eqref{osgood}. Then, there exists a unique classical flow $X\colon[0,T]^2 \times \R^d \to \R^d$ associated to $b$, the semigroup property holds true, i.e.\ 
    \begin{equation}
        X(t,\tau, X(\tau, s, x)) = X(t,s, x) \ \ \ \forall t, \tau, s \in [0,T], \ x \in \R^d 
    \end{equation}
    and $X(t,s, \cdot)$ is a homeomorphism of $\R^d$. Moreover, if $b$ is incompressible, then $X(t,s,\cdot)$ is a measure preserving transformation for any $t,s \in [0,T]$.
\end{thm}

%%% BIBLIO %%%

% quando si cita, togliere issn, review, doi. 
% citare come [iniziale cognome-anno]


\end{document}